\documentclass[a4,12pt]{amsart}

\usepackage{stackrel}
\usepackage{amssymb}
\usepackage{amsmath}
\usepackage{amsthm}
\usepackage[pdftex]{graphicx}
\usepackage[font=small]{caption}
\usepackage{mathtools}
\usepackage{pdfsync}
\usepackage{graphicx}
\usepackage{color}
\usepackage{tikz-cd}
\usepackage{tikz}
\usetikzlibrary{positioning}
\usetikzlibrary{decorations.text}
\usetikzlibrary{decorations.pathmorphing}
\usetikzlibrary{matrix}

\usepackage[all]{xy}

\newtheorem{theorem}{Theorem}[section]
\newtheorem{proposition}[theorem]{Proposition}
\newtheorem{lemma}[theorem]{Lemma}
\newtheorem{corollary}[theorem]{Corollary}

\theoremstyle{definition}

\newtheorem{example}[theorem]{Example}
\newtheorem{remark}[theorem]{Remark}

\numberwithin{equation}{section}
\DeclarePairedDelimiterX\Set[2]{\lbrace}{\rbrace}
 { #1 \,\delimsize|\,\mathopen{} #2 }

\def\({\left(}
\def\){\right)}

 \newcommand{\Z}{\mathbb{Z}}
  \newcommand{\N}{\mathbb{N}}
  \newcommand{\Q}{\mathbb{Q}}

 \newcommand{\tr}{\text{\bf tr}}

\title[]{Obstacles  to Topological Factoring of Toeplitz Shifts}

\author{Maryam Hosseini and Reem Yassawi}

\address{Department of Mathematics, School of Science and Technology, City St. George University of London, UK}
\email{maryam.hosseini@city.ac.uk}
\address{School of Mathematical Sciences, Queen Mary University of London, Mile End Road,
London E1 4NS, UK}
 \email{r.yassawi@qmul.ac.uk}

\begin{document}
\maketitle

\begin{abstract}
For every Toeplitz sequence $x$ with  period structure $(q_i)_{i\geq 1}$, one can identify a period structure ${\bf p}=(p_i)_{i\geq 0}$ which leads to a Bratteli-Vershik realization of the associated Toeplitz shift;  we refer to this period structure as {\it constructive}. Let $(X,\sigma,x)$ and $(Y,\sigma,y)$ be Toeplitz shifts where $x\in X$ and $y\in Y$ are Toeplitz sequences  with constructive period structures $(p^n)_{n\geq 1}$ and $(q^n)_{n\geq 1}$, respectively. Using the Bratteli-Vershik realization of factor maps between Toeplitz shifts, we prove that if there exists a topological factoring $ \pi:(X,\sigma)\rightarrow (Y,\sigma)$ with  $\pi(x)=y$, then $q\mid p$. In particular, if $\pi$ is conjugacy, then $p=q$. We also prove that Toeplitz sequences are mapped to Toeplitz sequences through topological factorings.
\end{abstract} 
%%%%%%%%%%%%%%%%%%%%%%%%%%%%%%%%%%%%%%%%%%%%%%%%%
\section{Introduction}
Toeplitz shifts are  topological dynamical systems that serve as almost one-to-one extensions of odometers, which are rotations on the topological group dual to a subgroup of complex roots of unity. Building on  Williams' work \cite{Williams}, studies on Toeplitz shifts uncovered  a wide range of dynamical behaviors; they can be constructed to support any metrizable Choquet simplex of measures \cite{Downarowicz-1991, GJ, Ormes-1997}, exhibit any measurable spectrum \cite{Downarowicz-Lacroix-1996}, or posses  any possible entropy \cite{Downarowicz-Serafin}.
 
 There have been extensive studies of automorphisms of  Toeplitz shifts \cite{{BK-1990}, {BK-1992}, {Cortez-Petite-2020}, {CQY}, {Donoso-Durand-Maass-Petite-2017},  {Dymek-2017},  {Dymek-Kasjan-Keller-2024}} and their possible (expansive) factors  \cite{ {Downarowicz-1997}, {down2},  {downdur}}, which must themselves be Toeplitz. In this article we are interested in conditions which are necessary for the existence of topological factoring between Toeplitz shifts. We are motivated by Cobham's theorem \cite{Cobham}, which states that   if $p$ and $q$ are multiplicatively independent, then only eventually periodic sequences can be both $p$- and $q$-{\em automatic}.   Durand extended  Cobham's theorem to codings of all {\it  primitive} substitutional sequences \cite{D}, (see Section~\ref{subs} for the relevant definitions). There is also a version of  Cobham's theorem for  {\em Mahler } sequences,  which are a generalisation of automatic sequences to infinite alphabets \cite{AB13, SS-2019}, the latter giving a new proof of Cobham's theorem, as does the recent article by Krebs \cite{krebs}. A dynamical statement of  Cobham's proved by Coven, Dykstra and Lemasurier  shows that shift dynamical systems generated  by primitive substitutions of multiplicatively independent lengths cannot be topologically conjugate \cite{CDL-2014}; this implies Cobham's theorem for the substitutional Toeplitz sequences belonging to the state space. In this article we study extension of this result to Toeplitz sequences.

Toeplitz shifts $(X,\sigma)$  contain  sequences  $x=(x_n)_{n\in \mathbb Z}$ which are tiled by monochromatic arithmetic progressions, i.e., each $n$ belongs to an arithmetic progression $I$ such that  $x_n=x_m$ for each $m\in I$. Such a sequence $x$ is  called {\em Toeplitz}, and the ratios  of the arithmetic progressions is called  a {\em period structure} $(p_n)_{n\geq 1}$ for $x$. 

Period structures for Toeplitz sequences are not unique, but each period structure determines the same  underlying odometer space: it is the dual of the group $G= G(X)$ generated by the  $k$-th roots of unity where $k$ is any factor of $p_1\dots p_N$, $N\in \mathbb N$. 
This group $G$ is a topological invariant for Toeplitz shifts, as it is the group of continuous eigenvalues of the operator $f\mapsto f\circ\sigma$ on the space of continuous complex valued functions on $X$. From this it follows that if $(Y,\sigma)$ is a topological factor of $(X,\sigma)$, then $G(Y)\subset  G(X)$. In other words, topological factoring can only occur when the prime divisors  of a period structure for a Toeplitz sequence in $Y$ are contained in the prime divisors of a period structure for a Toeplitz sequence in $X$.   But even if the prime factors of period structures are exactly the same, we may detect  further obstacles to factoring.  Our main result in this direction is the following theorem.
\begin{theorem}\label{thm:main}
Let $(X,\sigma)$ and $(Y, \sigma)$ be two  Toeplitz shifts. Suppose  that  $x$ and $y$ are Toeplitz sequences in $X$ and $Y$ with constructive period structures $(p^i)_{i\geq 1}$  and $(q^i)_{i\geq 1}$,  respectively. If there exists a topological factor map $\pi:(X,\sigma)\rightarrow (Y,\sigma)$ with $\pi(x)=y$, then $q\mid p$.  In particular, if $\pi$ is a conjugacy  between the two systems mapping $x$ to $y$, then $p=q$.  
\end{theorem}

 The proof of this theorem  relies on the result of \cite{aeg21} developed in a recent work of the first author with  Golestani and Oghli  \cite{gho} which give necessary and sufficient conditions for factoring between zero dimensional systems as a  generalization of  Curtis-Hedlund-Lyndon theorem \cite{Lind-Marcus}. There is an analogous statement given by Espinoza in
\cite{espinoza} for $S$-adic representations of minimal shifts. For the proof we first use the fact that  every Toeplitz sequence has a period structure  that we call it {\it constructive}, i.e., it satisfies the conditions of Lemma \ref{autol}. Constructive period structures have already been used by Gjerde and Johansen in \cite{GJ} to 
create Bratteli-Vershik realizations of  Toeplitz shifts.
Using the Bratteli-Vershik realizations of two Toeplitz shifts $(X,\sigma,x)$ and $(Y,\sigma,y)$,  we realize any factor map  $\pi:(X,\sigma)\rightarrow (Y,\sigma)$, $\pi(x)=y$,  by   a sequence of {\it morphisms} ${\bf \eta}=(\eta_n)_{n\geq 0}$ between ``appropriate'' levels of  the  associated ordered Bratteli diagrams. We elaborate this in Section \ref{model}, based on \cite[Section 5]{gho}.   Using the fact that we are representing Toeplitz shifts, we successively constrain the levels at which the morphisms occur in   Lemmas \ref{firstlevel}, \ref{main},  and \ref{proper}.  Then having the morphisms occurred at optimal levels, we show that the absence of certain conditions precludes the existence of topological factoring and consequently, conjugacy between two Toeplitz shifts with period structures as in Theorem \ref{thm:main}.  As stated in  Proposition \ref{general1}, a direct consequence of Lemma \ref{proper} is   that if there exists a topological factoring $\pi: (X,\sigma)\rightarrow (Y,\sigma)$ between Toeplitz shifts with general constructive period structures $(p_i)_{i\geq 1}$ and $(q_i)_{i\geq 1}$,  then there exists $i_0\geq 0$ such that for every $i \geq 1$, $q_i \mid p_{i_0 + i}$. Moreover, $\pi$ preserves the Toeplitz nature of any Toeplitz sequence $x \in (X,\sigma)$ when mapped into $(Y,\sigma)$, see Proposition \ref{general2}.
 As a corollary of Theorem 1.1,  we get Cobham's original result for constant length substitutions with coincidence; see Corollary \ref{Cobham}. 

\section{Prelimnaries}\label{pre}
\subsection{Dynamical systems }
A {\it topological dynamical system} is a pair $(X,T)$ such that $X$ is a compact metric space and $T:X\rightarrow X$ is a homeomorphism. The orbit of a point $x\in X$ is $\mathcal O(x)=\{T^nx\}_{n\in\Z}$. The system $(X,T)$ is called {\it minimal} if for every $x\in X$, $\overline{\mathcal O(x)}=X$.  A point $x\in X$ is {\it quasi-periodic} if  for any open neighbourhood $U\ni x$,  there exists $p>0$ such that for every $n\geq 1$, $T^{np}(x)\in U$ \cite{kurka}.  A dynamical   system $(Y, S)$ is called a {\it topological factor} of  $(X, T)$ if there exists a continuous onto map $\pi:X\rightarrow Y$ so that $\pi\circ T=S\circ \pi$. An equicontinuous factor is {\it maximal} if all other equicontinuous systems factor through it.
A {\it Cantor} dynamical system is a dynamical system such that $X$ is the Cantor set. A complex number $\lambda=\exp(2\pi it)$ is a continuous {\it eigenvalue} for $(X,T)$ if there exists a continuous function $f:X\rightarrow \mathbb S^1$ so that $f\circ T=\lambda f$. The set of all $t\in [0,1)$ such that $\exp(2\pi it)$ is an eigenvalue for $(X, T)$, is an additive group and is called the {\it continuous spectrum} of $(X,T)$ and is denoted by $E(X,T)$. The {\it rational spectrum} of $(X, T)$ is $E(X,T)\cap \Q$.  A system such that $E(X,T)\cap \Q$ is non-cyclic admits an equicontinuous factor $(\Z_{\bf p},+1)$ called an {\it odometer}, whose space is the topological group  $\Z_{\bf p}$ defined as follows.  Let ${\bf p}=(p_i)_{i\geq 0}$ be a sequence of natural numbers such that  $p_0=1$ and for every $i\geq 1$, $p_i\mid p_{i+1}$. The $\bf p$-odometer  denoted by $\Z_{\bf p}$, is defined as the inverse limit of the cyclic groups $ \Z/p_n\Z$,
$$ \Z_{\bf p}  =  \varprojlim \Z/p_n\Z, $$
with the natural projection maps  $\phi_n: \Z_{\bf p} \rightarrow   \Z/p_n\Z$. The map $+1$, acting on $\Z_{\bf p}$, is simply addition  of unity.

\subsection{Symbolic dynamics}
We recall some of the basics about the symbolic  dynamical systems that will be used in the sequel. 
We refer the reader to \cite{kurka, Lind-Marcus} for further details. Let $A$ be a finite set, called an {\it alphabet}, and $A^*$ be the set of all finite words on $A$.  Let $A^\Z$ be the set of all bi-infinite sequences $(x_n)_{n\in\Z}$ such that for every $n\in\Z$, $x_n\in A$. A {\it factor} or a {\it subword} of $x\in A^*$ is a word $w$, written as $w\sqsubset x$,  such that for some $1\leq i\leq |x|$, $x_ix_{i+1}\cdots x_{i+|w|-1}=w$. Equipped  with the product topology, $A^\Z$ is a compact metric space which is homeomorphic to the Cantor set. Then $A^\Z$ is called the {\it shift space}. The {\it shift map} is $\sigma: A^\Z\rightarrow A^\Z$  such  that for every $x\in {A^\Z}$, and every $n\in\Z$, $(\sigma(x))_n=x_{n+1}$. If $X\subseteq A^\Z$ is a closed and {\it $\sigma$-invariant}, i.e., $\sigma(X)=X$, then $(X,\sigma)$ is called a {\it shift} . A  system $(X,T)$, with $X$ a metric space, is {\em expansive} if there is some $\delta>0$ such that for any pair of distinct points $x, \tilde x$ in  $X$, there is an $n$ such  that $ d(T^n(x), T^n(y))\geq \delta$. A  Cantor system is expansive if and only if it is conjugate to a shift \cite{kurka}.

\subsubsection{Toeplitz shifts}\label{toeplitz}
Here we recall some of the basics about Toeplitz shifts. We refer the reader to \cite{Downarowicz2005, GJ, kurka} for more details. A non-periodic and quasi-periodic sequence in a shift system $(A^\Z,\sigma)$  is called {\it Toeplitz}.
Let  $x=(x_n)_{n\in\Z}$ be a non-periodic sequence in a shift space $A^\Z$ and consider $p\in\N$ and $a\in A$. Suppose that
$${\rm Per}_{p}(x,a):=\{n\in\Z: \ x_m=a\  \forall \ m\equiv n({\rm mod}\ p)\}$$
is non-empty.  Let ${\rm Per}_{p}(x):=\bigcup_{a\in A}{\rm Per}_{p}(x,p)$. Given $p$, define the {\em $p$-skeleton} of $x$ to be the sequence obtained from $x$ by replacing $x_n$ with a new symbol $*$ for $n\not\in {\rm Per}_{p}(x)$. We say that $p$ is an {\em essential period} for $x$ if the $p$-skeleton of $x$ is not periodic with any smaller period.  A {\it period structure} for $x$ is  a sequence of positive integers ${\bf p}=(p_i)_{i\geq 0}$ such that  
\begin{itemize}
\item $p_i$ is an essential period for $x$,
\item for every $i\geq 0$, $p_i\mid p_{i+1}$, and
\item  $\bigcup_{i=0}^\infty {\rm Per}_{p_i}(x)=\Z.$
\end{itemize}
An aperiodic sequence $x\in A^\Z$ with a period structure is called a {\it Toeplitz sequence}.   For a Toeplitz sequence $x\in A^\Z$, the shift $(X,\sigma)$ in which $X=\overline{\mathcal O}(x)$,  is called a {\it Toeplitz shift} and it is a minimal dynamical system.
The period structure of a Toeplitz sequence $x$  is not unique as the least common multiple of two essential period is an essential period. However, any period structure ${\bf p}=(p_i)_{i\geq 0}$ will generate the maximal equicontinuous factor of the Toeplitz shift generated by $x$ which is $(\Z_{\bf P},+1)$. The  {\em essential period} of a word $B=x_0\cdots x_\ell$ which appears in the Toeplitz sequence $x$ is the least common multiple  of the essential periods of $x_i$, $i=1,\ldots, \ell$. We will use the following lemma, which is proved in \cite[Lemma 6 and Theorem 8]{GJ}.
\begin{lemma}\label{autol}
Let $z$ be a Toeplitz sequence. Then there exist  a period structure
  ${\bf p}=\(p_i\)_{i\geq 1}$ for $z$ such that for every $i\geq 1$, the initial block $B_i:=z_{[0,p_i-1]}$ of $z$ has  essential period equal to $p_{i+1}$. 
 \end{lemma}
We say that the period structure ${\bf p}=\(p_i\)_{i\geq 1}$ is {\it constructive} for the Toeplitz sequence $z$,  if it satisfies the condition of Lemma \ref{autol}.  Based on the proof of Theorem 8 of \cite{GJ}, for every period structure $(q_i)_{i\geq 1}$ of a Toeplitz sequence $x$,  there exists a constructive period structure $(p_i)_{i\geq 1}$.

\subsubsection{Substitutional dynamical systems.}\label{subs}
A well studied subfamily of Toeplitz shifts are {\em substitutional};  we briefly define these.  
Let $A$ be a finite {alphabet}.  A morphism $\theta:A^*\rightarrow A^*$ is called a {\it substitution}. Extending $\theta$ by concatenation, for  $\theta:A^{\Z}\rightarrow A^{\Z}$ a point $x\in A^{\Z}$ is called a fixed point of $\theta$ if $\theta(x)=x$.  For each $n\geq 0$, $\theta^n:A^*\rightarrow A^*$, that is the n-th iteration of $\theta$, is again a substitution. The substitution $\theta$ is  \emph{primitive} if there is a $k$ such that every $b\in A$ appears in $\theta^k(a)$ for each $a\in A$.   The substitution  $\theta$  is a (constant)  {\it length}-$\ell$ substitution if the length of
the word $\theta(a)$  equals  $\ell$ for every $a\in A$.
A constant length substitution $\theta$ has a {\it coincidence} if there exists some $n\geq 0$ and some $i$, $1\leq i\leq |\theta^n|$ such that the  $i$-th letter of $\theta^n(a)$ is a fixed letter $b\in A$ for all $a\in A$.
Given a substitution $\theta\colon A \rightarrow A^+$, the  language   $\mathcal{L}_{\theta}$    defined by  $\theta$ is
$$\mathcal{L}_{\theta} = \big\{w \in A^*:\, \mbox{$w$ is a subword of $\theta^n(a)$ for some  $a\in A$ and $n\in \mathbb N$}\big\}.$$
If $\theta$ is primitive then each word in $\mathcal{L}_{\theta}$ is left- and right-extendable, and $\mathcal{L}_{\theta}$ is closed under the taking of subwords, so this language defines a shift  $(X_\theta,\sigma)$, called the {\it substitutional shift} generated by  $\theta$, see \cite{Queffelec}. The shift associated to a primitive substitution is minimal. If $(X_\theta,\sigma)$ has no shift-periodic points then it is called \emph{aperiodic}.  In this article the substitutions we consider, are of constant length,  primitive, and aperiodic. If $\theta$ is primitive and aperiodic, then Dekking's theorem \cite{dekking} tells us that  $(X_\theta, \sigma)$ is a finite-to-one extension of  $(\Z_{\bf p},+1)$ where ${\bf p}=(\ell^ih)_{i\geq 1}$ for some $h\geq 1$ coprime with $\ell$. Furthermore, $(X_\theta, \sigma)$ is Toeplitz if and only if  $\pi: X_\theta\rightarrow  \Z_{\bf p}$ satisfies $\min_{z\in \Z_{\bf p}} |\pi^{-1}(z)|=h$. Dekking shows that $(X_\theta, \sigma)$ is Toeplitz
 if and only if $\theta$ has a  coincidence. He also  shows that  $(X_\theta, \sigma)$ has a constant height $h$ suspension over another length-$\ell$ substitution shift $(X_\theta, \sigma)$ which is its {\em pure base}, and which has a coincidence and height $h=1$.

Let $\theta$ be a length-$\ell$ substitution. A sequence $u$ is $\ell$-automatic if  $u=\tau (v)$, where $\tau: B\rightarrow A$ is a code which is the local rule of a shift commuting map from $B^\Z$ to $A^\Z$, and $v$ is a $\theta$-periodic point  with $\theta$ a length-$\ell$ substitution. We will write that  $u$ is generated by $(\tau,\theta)$.
 Cobham's theorem \cite{Cobham} tells us that if a sequence $v$ is both $k$- and $\ell$-automatic for multiplicatively independent $k$ and $\ell$, then $v$ is eventually shift-periodic. Any topological factor of a length-$\ell $ substitutional shift is topologically conjugate to a shift generated by an $\ell$-automatic sequence, and  the shift generated by an $\ell$-automatic sequence is 
topologically conjugate to a length-$\ell$ substitution shift \cite{MY}.

\subsection{Kakutani-Rokhlin partitions}\label{KR}
Let $(X,T)$ be a minimal Cantor system and consider a point $z\in X$. Let $(U_i)_{i\geq 0}$ with $U_0=X$, be a sequence of clopen sets converging to $z$. Using the  {\it first return time map,} of $U_i$s there exists a sequence of Kakutani-Rokhlin (K-R) partitions  $(\mathcal P_n)_{n\geq 0}$ of clopen sets that generate the topology such that  $\mathcal P_0=X$ and
$$ \mathcal P_n=\bigsqcup_{i=1}^{k(n)} \bigsqcup_{j=0}^{h_i (n)-1} B_{ij},\ \ B(\mathcal P_i)=\bigsqcup_{i=1}^{k(n)} B_{i0}=U_i, \ \ T(B_{ij})=B_{ij+1}.$$
For details see \cite{hps92, putnam}.
 
Consider the Toeplitz shift  $(X,\sigma)$ where $X=\overline{\mathcal O(z)}$ where $z$ is the Toeplitz sequence with constructive period structure ${\bf p}=\(p_i\)_{i\geq 1}$. Then for every $n\geq 1$, the initial block $B_n:=z_{[0,p_n-1]}$ is the base of the K-R partition  $(\mathcal P_n)_{n\geq 0}$, see \cite{DHS, GJ}.

\subsection{Bratteli–Vershik realizations of zero dimensional systems}\label{versh}
By  \cite{gps95,  hps92} having a sequence of K-R partitions for the Cantor minimal system $(X,T)$ we can realize the system by a dynamical system on an infinite graph which is called a Bratteli-Vershik realization. We recall some basics about ordered Bratteli diagrams and Vershik systems.

%\begin{definition}\label{premorphism}
 A 
\emph{Bratteli diagram}
$B=(V, E)$ 
consists of an infinite sequence of finite, non-empty, pairwise disjoint sets 
$V_{0} = \{ v_{0} \}, V_{1}, V_{2}, \ldots$, 
called the vertices, another sequence of finite, non-empty, pairwise disjoint sets 
$E_{1}, E_{2}, \ldots$, 
called the 
\emph{edges}, 
and two maps 
$s : E_{n} \rightarrow V_{n-1}, r : E_{n} \rightarrow V_{n}$, 
for every 
$n \geq 1$, 
called the range and source maps, such that 
$r^{-1} ( v )$ 
is non-empty for all 
$v$ 
in 
$\cup_{n \geq 1} V_{n}$ 
and 
$s^{-1} ( v )$ 
is non-empty for all 
$v$ 
in 
$\sqcup_{n \geq 0} V_{n}$. For every $n\geq 1$ we have an {\it adjacency}  matrix $M_n$ of size $|V_n|\times |V_{n-1}|$ that its entries $M_n^{ij}$ shows the number of edges between $v_i\in V_n$ and $v_j\in V_{n-1}$. 
A Bratteli diagram is called {\it  simple} if there exists some sequence $\(n_k\)_{k\geq 1}$ so that $$\forall k\geq 1\ \ M_{n_k}\cdot M_{n_k+1}\cdots M_{n_{k+1}}>0.$$
An 
\emph{ordered Bratteli diagram}
$B=(V, E, \leq)$
consists of a Bratteli diagram 
$(V, E)$ 
and a partial order
$\leq$
on 
$E$ 
such that two edges 
$e, e^{\prime}$ 
in 
$E$ 
are comparable if and only if 
$r(e) = r(e^{\prime})$. The partial ordering can be extended to the set of all finite paths from $v_0$ that terminate at the same vertex $v\in V_i$ for some $i\geq 1$.
In such a diagram, we let 
$E_{\max}$ 
and 
$E_{\min}$ 
denote the set of maximal and minimal edges, respectively. For every $i\geq 1$, $h_j^i$ denotes the number of finite paths from $V$ to the vertex $v_j\in V_i$.

 Let  $B=(V,E,\leq)$ be an ordered Bratteli diagram. Then $(\theta_i)_{i\geq 1}$, $\theta_i:V_i\rightarrow V_{i-1}^*$ is defined  by 
$$ \theta_i(v)=s(e_1(v))s(e_2(v))\cdots s(e_k(v)),\ \ {\rm for} \ i\geq 2$$
where $\{e_j(v):\ j=1,\ldots, k(v)\}$ is the ordered set of the edges in $E_i$ with range $v$, and for $i=1$, $ \theta_1:V_1^*\rightarrow E_1^*$, is defined by $\theta_1(v)=e_1(v)\cdots e_\ell(v)$ where $e_1(v),\ldots, e_\ell(v)$ are all the edges in $E_1$ with range $v\in V_1$ and $e_1(v)<\cdots<e_\ell(v)$. Note that by concatenation, one can extend $\theta_i$ as 
$\theta_i:V_i^*\rightarrow V_{i-1}^*$. Also, $\theta_{(i,j]}=\theta_{i+1}\circ\theta_{i+2}\circ\cdots\circ\theta_j$ is a morphism from $V_j^*$ to $V_i^*$ for $0\leq i\leq j$. We say that a morphism $\theta:A^*\rightarrow B^*$ is letter-surjective if for any $b\in B$ there is $a\in A$ such that $b$ appears in $\theta(a)$.  As the morphisms $(\theta_i)_{i\geq 1}$ provide complete information about the order; when needed, we will write $B=(V, E,(\theta_i)_{i\geq 1})$ instead of  $B=(V,E,\leq)$.

A {\it telescoping} on an ordered Bratteli diagram $B=(V,E,(\theta_i)_{i\geq 1})$ involves choosing an increasing sequence of natural integers $\(m_i\)_{i\geq 1}$ to construct a new ordered Bratteli diagram $B'=(V',E',(\theta'_i)_{i\geq 1})$, where $V'_0=V_0$ and for each $i\geq 1$, $V'_i=V_{m_i}$ and $\theta'_i=\theta_{m_i}\circ\theta_{m_i-1}\circ\cdots\circ\theta_{m_{i-1}+1}$. 
There is an obvious notion of {\it isomorphism} between two ordered Bratteli diagrams. Two ordered Bratteli diagrams are {\it equivalent} if they are isomorphic or can be derived from one another by telescoping.

Let  $B=(V,E,\leq)$ 
be an ordered Bratteli diagram and $X_B$ be the compact space  of all infinite paths (originated from  $v_0\in V_0$), on $B$.  Let $X_B^{\rm max}$ denote the set of all {\em infinite maximal paths} in $X_B$, i.e., the elements of $X_B$ each of whose component edges belonging to $E_{\max}$, and similarly let $X_B^{\rm min}$ denote the set of all {\it infinite minimal paths} in $X_B$. 

The {\em Vershik} map $T_B:X_B\setminus X_B^{\rm max} \rightarrow X_B\setminus X_B^{\rm min}$ is defined by
$$T_B(e_0, e_1, \ldots, e_\ell,e_{\ell+1},\ldots)=(0,0,\ldots, e_{\ell}+1,e_{\ell+1},\ldots)$$
where $\ell$ is the first index that $e_\ell$ is not the max edge in $r^{-1}(r(e_{\ell}))$ and  $0$ denotes the minimal edge in $r^{-1}(r(e_i))$  for every  $i\geq 0$. 
$B$  is called \emph{properly ordered} if it has a unique infinite minimal path $x_{\min}$ and a unique infinite maximal path $x_{\max}$. Then the Vershik map can be extended to be a homeomorphism on $X_B$, that is $T_B: X_B\rightarrow X_B$ maps  the unique maximal path to the unique minimal path. Note that we can telescope a properly ordered diagram so that at each level, the maximal edge has the same source, and the minimal edge has the same source.

When all the adjacency matrices  have equal row sums, the Bratteli diagram is called {\it ERS}. In this note we only consider  simple properly ordered ERS  Bratteli diagrams.  A constructive period structure  is used to prove the following important result.
\begin{theorem}\cite[Theorem 8]{GJ}\label{auto}
The family of all expansive  Vershik systems on  properly ordered ERS Bratteli diagrams coincides with the family of  Toeplitz shifts up to conjugacy.
\end{theorem}
Furthermore, inspection of the proof of \cite[Theorem 8]{GJ} tells us that  the ERS Bratteli-Vershik representation of  a Toeplitz shift can be taken  so that the row sum of the $n$-th adjacency matrix equals $p_n/p_{n-1}$ ($p_0=1$), where $(p_n)$ is a constructive  period structure,  and such that for each $n$ all minimal (resp. maximal) edges at level $n$ have the same source.

Consider the Toeplitz shift $(X,\sigma)$ with the Toeplitz sequence $z\in X$ and the realization of $(X,\sigma)$ by the sequence of K-R partitions $\(\mathcal P_k\)_{k\geq 0}$ so that $\bigcap_{k\geq 0}B(\mathcal P_k)=\{z\}$. 

 Let $B=(V, E,\leq)$, $V=\sqcup_{k\geq 0}V_k$ be the associated Bratteli-Vershik model which means that for every $k$, each vertex in $V_k$ is a representative of a tower in $\mathcal P_k$. So $|V_k|$ is equal to the number of towers in $\mathcal P_k$. A vertex $v\in V_k$ is connected (via an edge in $E_k$) to a vertex $w\in V_{k-1}$ if the tower in $\mathcal P_{k-1}$ which is associated with $w$, appears as a sub-tower of  that tower in $\mathcal P_k$ which is associated with $v$. Moreover,  the unique infinite minimal path $x_{\min}\in X_B$  is associated with $z$. 
So for every $k$, one can  consider $\mathcal P_k$ as the set of all finite paths from $V_0$ to $V_k$.  Then there exists a truncation map $\tau_k:X_B\rightarrow \mathcal P_k$ that restricts each infinite path in $X_B$ to its initial segment of length $k$. So the natural projections are $\tilde{\tau}_k: X_B\rightarrow \mathcal P_k^\mathbb Z$ that are defined by 
 \begin{equation}\label{proj}
 \tilde{\tau}_k(x)=(\tau_k(T_B^nx))_{n\geq 0}.
 \end{equation}
  It turns out that at each level $k$, we have a shift:
$$(\tilde{\mathcal P}_k,\sigma) \ \ {\rm where} \ \ \tilde{\mathcal P}_k=\tilde{\tau}_k(X_B)\subseteq \mathcal P_k^\mathbb Z, \ \ \sigma(\tilde{\tau}_k(x))=\tilde{\tau}_k(T_Bx).$$ 
Following  \cite{GJ}, we denote $(\tilde{\mathcal P}_k,\sigma)$ by $(X_k,\sigma)$ and in the sequel we may recall them by {\it the intermediate shifts associated with} $(X,\sigma)$. Note that by the definition of $\tilde{\mathcal P}_k$, the map $\tilde{\tau}_k:X_B\rightarrow X_k$ is  a topological factoring 
\begin{equation}\label{trunc}
\tilde\tau_k\circ T_B=\sigma\circ\tilde\tau_k.
\end{equation}
Let $\tr_k: (\mathcal P_k,\sigma)\rightarrow (\mathcal P_{k-1},\sigma)$ be the natural map that sends a path of length $k$ (originated at $V_0$)  to its initial path of length $k-1$, and extend $\tr_k$ component-wise to a map $\tilde{\tr}_k: (\tilde{\mathcal P}_k,\sigma)\rightarrow (\tilde{\mathcal P}_{k-1},\sigma)$; then  for each $k$, $\tilde{\tr}_k$ commutes with the shift, i.e., is a {\em sliding block code with radius 0}, and  $\tilde{\tr}_k\circ\tilde\tau_{k}=\tilde\tau_{k-1}$ for all $k\geq 1$. Consequently, 
\begin{equation}\label{inlim}
 \xymatrix{(X_0,T_0)
&(X_1,\sigma)\ar[l]_{\tilde\tr_{1}} &(X_2,\sigma)\ar[l]_{\tilde\tr_{2}} &\cdots\ar[l]_{\tilde\tr_3}& (X_B,T_B)\ar[l]
}
\end{equation}
$$(X,\sigma)\simeq(X_B,T_B)=\varprojlim_{k}(X_k,\sigma).$$

\subsection{Topological factoring and Bratteli-Vershik realizations}\label{model}
Using the inverse limit realization of Cantor minimal systems in (\ref{inlim}), one can model the topological factoring between two such systems as the limit  of a sequence of sliding block codes between the intermediate shifts associated with the two systems. This has been verified in \cite{gho} for the general case of zero dimensional dynamical systems. In the following theorem we state that for the special case of Cantor minimal systems. 
\begin{theorem}\cite[Theorem 1.3]{gho}\label{inversesys}
Consider minimal Cantor  dynamical systems $(X,T)$ and $(Y, S)$.
Let $x_0\in X$ and $y_0\in Y$. 
Then there exists
$ \pi:(X,T)\longrightarrow (Y,S)\ \ {\rm with} \ \  { \pi(x_0)=y_0}$
if and only if  for every  pair of sequences of K-R partitions  $\(\mathcal P_n\)_{n\geq 0}$ and $\(\mathcal Q_n\)_{n\geq 0}$ for $(Y,S,\{y_0\})$ and $(X,T,\{x_0\})$ respectively, and the inverse limit systems associated to them as in (\ref{inlim}),  there exists a sequence of natural numbers $\(n_i\)_{i\geq 0}$ and a sequence of radius $0$ sliding block codes $\pi_i:(X_{n_i},\sigma)\rightarrow (Y_i,\sigma)$  for all $i\geq 0$  such that 
 all of the rectangles   between the following  inverse limit sequences commute:
 \begin{equation*}%\label{diag}
\xymatrix{(X_{0},\sigma)\ar[d]_{\pi_{0}}
&(X_{n_1},\sigma)\ar[l]_{_{\tilde\tr_{(0,n_1]}}}\ar[d]_{\pi_{1}} &(X_{n_2},\sigma)\ar[l]_{\tilde\tr_{(n_1,n_2]}}\ar[d]_{\pi_{2}} &\cdots\ar[l]_{\tilde\tr_{(n_2,n_3]}}& (X,T,\{x_0\})\ar[l]\ar[d]_{\pi}\ \ \  \\
(Y_{0},\sigma)
&(Y_1,\sigma)\ar[l]^{\tilde\tr_1} &(Y_{2},\sigma)\ar[l]^{\tilde\tr_2}&\cdots\ar[l]^{\tilde\tr_3}&(Y,S, \{y_0\})\ar[l] \   
 }
\end{equation*}
where $\tilde\tr_{(i,j]}:=\tilde\tr_{i+1}\circ\tilde\tr_{i+2}\circ\cdots\circ\tilde\tr_{{j}}$ and $\tilde\tr_i$'s are the local truncation maps as in (\ref{inlim}) for either $(X,T)$ or $(Y,S)$. 
\end{theorem}
On the other hand, there is a model for topological factoring between two Cantor minimal systems called {\it ordered premorphism}  introduced in \cite{aeg21}. The concept of ordered premorphism has been reinterpreted in \cite{gho} as a sequence of morphisms between two ordered Bratteli diagrams. This sequence of morphisms  ``localizes"  the factoring map $\pi:(X,T)\rightarrow (Y,S)$ at some specific  levels of the Bratteli diagrams  associated to the two systems, by intertwining the finite collections of (finite) paths restricted to those  levels. In fact, by the results of \cite{aeg21, gho},  having such a sequence of  morphisms $(\eta_i)_{i\geq 0}$, is equivalent to  the existence of  a factoring map between the two systems.  We review the notation related to these morphisms and a key result about them from \cite{gho}.

Let $B=(W,E, \(\xi_k\)_{k\geq 1} )$ and $C=(V,E',\(\theta_k\)_{k\geq 1})$  be  two ordered Bratteli diagrams, and let  $\xi_{(\ell,k]}:=\xi_{\ell+1}\circ\xi_{\ell+2}\cdots\circ \xi_{k}$ for $k>\ell$.
Suppose that we have  a morphism  $\eta_k:W_{n_k}\rightarrow V_k^*$, for some $k\geq 1$, i.e.,  for each vertex $w\in W_{n_k}$ there exists some $m\geq 1$ such that $\eta_k(w)=v_{i_1}\cdots v_{i_m}\in V_k^*$.  So the morphism $\eta_k$ can be extended to $W_{n_k}^*$ by concatenation. 
\begin{proposition}\cite[Proposition 5.1]{gho}\label{factor1}
Let $(X,T)$ and $(Y,S)$ be Cantor minimal systems, and let $x_0\in X$ and $y_0\in Y$. There exists a topological factoring $\pi:X\rightarrow Y$ with $\pi(x_0)=y_0$ if and only if  
for each pair of Bratteli-Vershik models $B=(W,E,(\xi_i)_{i\geq 1})$ and $C=(V,E', (\theta_i)_{i\geq 1})$ for $(X,T,\{x_0\})$ and $(Y,S,\{y_0\})$ respectively,   there exist  an increasing sequence $\(n_i\)_{i\geq 0}$ of non-negative integers with $n_0=0$,  and a sequence $(\eta_i)_{i\geq 0}$ of non-erasing letter-surjective morphisms $\eta_i:W_{n_i}^*\rightarrow V_i^*$,  with $\eta_0(v_0)= w_0$ such that the following diagram commutes for all $i,j\geq 0$:
\begin{equation}\label{diagram}
\xymatrix{W_{n_i}^*\ar[d]_{\eta_{i}}
& \ \ W_{n_{j}}^*\ar[l]_{_{\xi_{(n_i,n_{j}]}}}\ar[d]_{\eta_{j}} &\ \  \\
\ V_{i}^*
& \ \ V_{j}^*.\ar[l]^{\theta_{(i,j]}} &\   
 }
\end{equation}
\end{proposition}\qed

In light of Proposition \ref{factor1}, when there is a factor map $\pi:(X_B,T_B)\rightarrow(X_C,T_C)$ between Vershik systems on properly ordered Bratteli diagrams, we will say that the factor map $\pi:X_B\rightarrow X_C$ is {\em realized} by $\eta=(\eta_i)_{i\geq 0}$ at levels $(n_i)_{i\geq 0}$. Note that $\eta$ sends the unique minimal path of $X_B$ to the unique minimal path of $X_C$.
Having a morphism $\eta_k:W_{n_k}\rightarrow V_k^*$, for each vertex $w\in W_{n_k}$ and $\eta_k(w)=v_{i_1}\cdots v_{i_m}\in V_k^*$, one can consider 
an ordered set of edges, say $\{g_1, g_2, \cdots, g_m\}$,
 such that for every $1\leq j\leq m$, $s(g_j)=w$ and $r(g_j)=v_{i_j}$, i.e. the source of the edge $g_j$ is $w$ and the range of that is $v_{i_j}$.  See Figures 1-5 as examples.

%%%%%%%%%%%%%%%%%%%%%%%%%%%%%%%%%%%%%%%%%%%%%
%%%%%%%%%%%%%%%%%%%%%%%%%%%%%%%%%%%%%%%%%%%%%
%%%%%%%%%%%%%%%%%%%%%%%%%%%%%%%%%%%%%%%%%%%%%
\section{Topological factoring and  ERS ordered Bratteli diagrams}\label{sec:main}
In this section, we study topological factoring between two Vershik systems on ERS ordered Bratteli diagrams. 

Suppose that $B=(W,E,  \(\xi_k\)_{k\geq 1}  )$ and $C=(V,E',   \(\theta_k\)_{k\geq 1}   )$ are two ERS ordered Bratteli diagrams such that $(X_B, T_B)$ and $(X_C, T_C)$ are Vershik systems.  So  the vertices of the first level of $B$ (resp. $C$)  are connected to $W_0=\{w_0\}$ (resp. $V_o=\{v_0\}$) each one by  $|\xi_1|$ (resp. $|\theta_1|$) edges. 
Suppose that $\pi: (X_B, T_B)\rightarrow (X_C, T_C)$ is a factor map which is realized by the sequence of  morphisms $\eta=\(\eta_i\)_{i\geq 1}$, $\eta_i:W_{n_i}\rightarrow V_i$  satisfying   (\ref{diagram}). 

 We first show that the sequence $\(n_i\)_{i\geq 0}$ can be optimized such that $n_i$ is the smallest natural number for which the heights of the towers in $\mathcal Q_{n_i}$ (associated with $W_{n_i}$) are linear combinations of the heights of the towers in $\mathcal P_i$ (associated with $V_i$). This is established for $n_1$ in Lemma \ref{firstlevel} and for $n_i$ when $i \geq 2$ in Lemma \ref{main}. We mention that for the proof of Lemma \ref{firstlevel}, we will use a technique similar to the {\em symbol splitting} method, originally described in \cite{gps95}, to create pairs of equivalent diagrams.

\begin{lemma}\label{firstlevel}
Suppose that   $(X_B, T_B)$ and $(X_C, T_C)$ are Vershik systems on the ERS properly ordered Bratteli diagrams $B=(W,E,  \(\xi_k\)_{k\geq 1})$ and $C=(V,E'', \(\theta_k\)_{k\geq 1})$. Let $\pi:(X_B, T_B)\rightarrow (X_C, T_C)$ be a factor map realized by $(n_i)_{i\geq 0}$ and  $\eta=\(\eta_i\)_{i\geq 0}$. If
 there exists $1\leq \ell<n_1$ such that $|\xi_{(0,\ell]}|=s|\theta_1|$ for some $s\geq 1$, then there exists a properly ordered Bratteli diagram
 $B'=(W', E',  \(\xi'_k\)_{k\geq 1})$ which is isomorphic to $B$  and whose vertex sets $(W'_n)_{n\geq 1}$ satisfy
  \[\mbox{    $W'_i=W_{\ell+i-1}$  for $i\geq n_1$.}\] Moreover,   there is  a  sequence of morphisms
   $\eta'=\(\eta'_i\)_{i\geq 0}$ such that 
    \[\eta'_0=\eta_0,\ \ \eta'_1:W'_1\rightarrow V_1^{*}   \mbox{ and  $\eta'_i=\eta_i$  for $i\geq 2$.}\]
  In particular, $\eta'$ induces the same factor map $\pi$.  
\end{lemma}
\begin{proof} 
The construction of the new diagram $B'=(W',(\xi'_i)_{i\geq 1})$ proceeds as  follows. We will  use a technique  to define  $W'_i$ and $\xi'_i$ for every $1\leq i\leq  n_1-\ell+1$ and then for every $i> n_1-\ell+1$ we just let $W'_i:=W_{i+\ell-1}$ and $\xi'_i:=\xi_{i+\ell-1}$.

 First consider all finite paths between $W_\ell$ and $W_{n_1}$. 
The maps $\eta_0$ and $\eta_1$ act as follows. 
 \[
\xymatrix{W^*_{0}\ar[d]_{\eta_{0}}
& W_\ell^*\ar[l]_{_{\xi_{(0,\ell]}}}&{W}_{\ell+1}^*\ar[l]_{_{\xi_{\ell+1}}}& \cdots\ar[l]_{\xi_{\ell+2}}& {W}_{n_1}^*\ar[l]_{\xi_{n_1}}\ar[llld]^{\eta_1}\\
\ V^*_{0}
&  V_{1}^*\ar[l]^{\theta_1} &&&
 }
 \] 
Fix a vertex ${w_1}\in W_\ell$; by assumption there are exactly $|\xi_{(0,\ell]}|=s|\theta_1|$ edges between ${w_1}$ and $W_0$. 
Now  for each $w\in {W}_{n_1}$, \eqref{diagram} tells us that $|\eta_{1}(w)|={|\xi_{(0,n_1]}|\over |\theta_1|}$. Thus
$$|\eta_{1}(w)|= {|\xi_{(\ell,n_1]}||\xi_{(0,\ell]}|\over |\theta_1|}={|\xi_{(\ell,n_1]}| (s |\theta_1|)\over |\theta_1|}=s|\xi_{(\ell,n_1]}|=s|\xi_{(\ell,n_1]}(w)|.$$
Consequently,  if there are $k$ ordered finite paths $f_1, \dots , f_k$ from $W_\ell$ whose ranges are the vertex $w\in W_{n_1}$, then we can write 
$\eta_1(w)$ as a concatenation of $k$ words $\eta_{1}^{(f_1)}(w), \dots , \eta_{1}^{(f_k)}(w)$, each of length $s$, i.e., 
\begin{equation}\label{concat}
 \eta_1(w)= \eta_{1}^{(f_1)}(w)\eta_{1}^{(f_2)}(w)\dots \eta_{1}^{(f_k)}(w). 
 \end{equation}
   See Figure \ref{fig1} and Example \ref{example:one} for an example. 
 Next we replace the vertex ${w_1}\in W_\ell$ with a new set  $\widehat{w_1}:=\{w'_{11}, w'_{12}, \ldots, w'_{1k({w_1})}\}$ of $k({w_1})$ vertices, where 
 \begin{equation}\label{cardinal}
 k({w_1})=|\{\eta_1^{(f)}(w):\ w\in {W}_{n_1},\ s(f)={w_1}, r(f)=w\}|.
 \end{equation}
We apply a similar procedure to each vertex in $W_\ell$ as we did for ${w_1}$ and define  
\begin{equation}\label{w1}
W'_1:=\bigcup_{i=1}^ {|W_\ell|} \widehat{w_i}.
\end{equation}
To guide the rest of this proof, Figure \ref{fig1} of Example \ref{example:one}  is continued in  Figure \ref{fig2}. Note that $|W'_1|$ may be strictly  less than the number of  finite paths between $W_\ell$ and ${W}_{n_1}$, as different finite paths $f$ and $f'$ with  the same source in $W_\ell$ and ranges  $w_1, w_2\in {W}_{n_1}$ may satisfy $\eta_1^{(f)}(w_1)=\eta_1^{(f')}(w_2)$, i.e.,  for each $1\leq i\leq |W_\ell|$, we look at $\{\eta_1^{(f)}(w):\ w\in {W}_{n_1},\ s(f)={w_i}, r(f)=w\}$ as a set. \\
For the ordered edge set $E'_1$,  from $W'_0$ to $W'_1$,  we  use the morphism  $\xi_{(0,\ell]}$ to define $\xi'_1: W'_1 \rightarrow W'_0$ by setting
    $$\xi'_1 (w') := \xi_{(0,\ell]}({w})\  \ {\rm for}\  w'\in{\widehat{w}}.$$
Next we define $W'_i$  inductively for $2 \leq i \leq n_1 - \ell$ using a similar procedure to that for $W'_1$. Then to define $\xi'_i$, we use similar argument as for  the definition of $\xi'_1$ but with some modifications with respect to $W'_{i-1}$.\\
 Fix $2 \leq i \leq n_1 - \ell$.  Recall that for  every vertex $w\in W_{n_1}$, $\eta_1(w)$ is a concatenation  of words as in (\ref{concat}).
 Note that each finite path $f_i$  from $W_\ell$ to $W_{n_1}$ passes through $W_{i+\ell-1}$, where each $f_i$ consists  of  a finite path from $W_{\ell}$ to $W_{i+\ell-1}$, followed by another one   from $W_{i+\ell-1}$ to $W_{n_1}$. So for a fixed vertex $w_j\in W_{i+\ell-1}$, $1\leq j\leq |W_{i+\ell-1}|$, every  finite path $g$ from $w_j$ to any vertex $w\in W_{n_1}$, continues $|\xi_{[\ell+1,i+\ell-1]}|$ finite paths from $W_\ell$ to $w_j$. 
 This means that  $\eta_1(w)$ is a concatenation of words  each of length 
$s|\xi_{[\ell+1, i+\ell-1]}|$.
Thus  $w_j$ is linked  to  a collection of words each of the form:
 $$\eta_1^{(g)}(w_j):=\eta_1^{(f_r)}(w)\eta_1^{(f_{r+1})}(w)\cdots\eta_1^{(f_{r+|\xi_{[\ell+1,i+\ell-1}|-1]})}(w),$$
 where $f_r, f_{r+1}, \ldots, f_{r+|\xi_{[\ell+1, i+\ell-1]}|-1}$ are the $|\xi_{[\ell+1,i+\ell-1]}|$ consecutive finite paths from $W_{\ell}$ to  $W_{n_1}$  through $w_j$, all sharing the segment from $w_j$ to  $W_{n_1}$ in $g$.

To define $W'_i$, for each  vertex  $w_j\in W_{i+\ell-1}$, $1\leq j\leq |W_{i+\ell-1}|$,   we consider all the finite paths from $w_j$ to $W_{n_1}$ and then we use the same notation as in (\ref{cardinal}) to define $k(w_j)$ and $\widehat{w_j}$. Consequently, 
\begin{equation}\label{w2}
W'_i:=\bigcup_{j=1}^ {|W_{i+\ell-1}|} \widehat{w_j}.
\end{equation}
Next,   we define $\xi'_i:W'_i\rightarrow W'_{i-1}$. Consider  $w'\in \widehat{w}_j\subset W'_i$ which  is linked to  the  word $\eta_1^{(g)}(w_j)$ . This word is a concatenation of   subwords of length 
\begin{itemize}
\item $s$ if $i=2$, and
\item $s|\xi_{[\ell+1, i+\ell-2]}|$ if $i>2$.
\end{itemize}
By the definition of $W'_{i-1}$, each  of these  subwords is linked to a vertex $w''\in\widehat{w_n}\in W'_{i-1}$, (where $w_n\in W_{i+\ell-2}$). 
So we place  an edge from $w''$ to $w'$ with an ordinal number corresponding to the order of its associated word as a subword of  $\eta_1^{(g)}(w_j)$. Specifically, if $w''$ is associated with the word $\eta_1^{(h)}(u)$ for some vertex $u\in W_{n_1}$ and some finite path  $h$  from $w_n$ to $u$, then 
$$w''\sqsubset \xi'_i(w')\ \Leftrightarrow \ \left (w_n\sqsubset \xi_{i+\ell-1}(w_m)\ {\rm and} \ \eta_1^{(h)}(u)\sqsubset \eta_1^{(g)}(w_j)\right ),$$
where $a\sqsubset b$ means that $a$ is a subword of $b$. Now let $W'_{n_1-\ell+1}:=W_{n_1}$ and define $\xi'_{n_1-\ell+1}: W'_{n_1-\ell+1}\rightarrow W'_{n_1-\ell}$  as follows. For any two vertices  $w\in W'_{n_1-\ell+1}$ and $w'\in \widehat{w_j}\subset W_{n_1-\ell}$  (where $w_j\in W_{n_1-1}$), we have $w'\sqsubset \xi'_{n_1-\ell+1}(w)$ if and only if
$$w_j\sqsubset \xi_{n_1}(w)\ {\rm and}\ \ \eta_1^{(e)}(w)\sqsubset \eta_1(w),\ {\rm for\ some}\ e\ {\rm connecting}\ w_j \ {\rm to}\ w.$$
   The sequence of sets $(W'_i)_{i\geq 0}$  and $(\xi'_i)_{i\geq 1}$, form the graded vertex sets with the partially ordered sets of edges $(E'_i)_{i\geq 1}$ for the diagram $B'$. It is not hard to check that $B'$ is isomorphic to $B$. \\
Note that since $B'$ is a finite modification of $B$, and $B$ is properly ordered, so is $B'$.
 \begin{figure}
\includegraphics[scale=0.2]{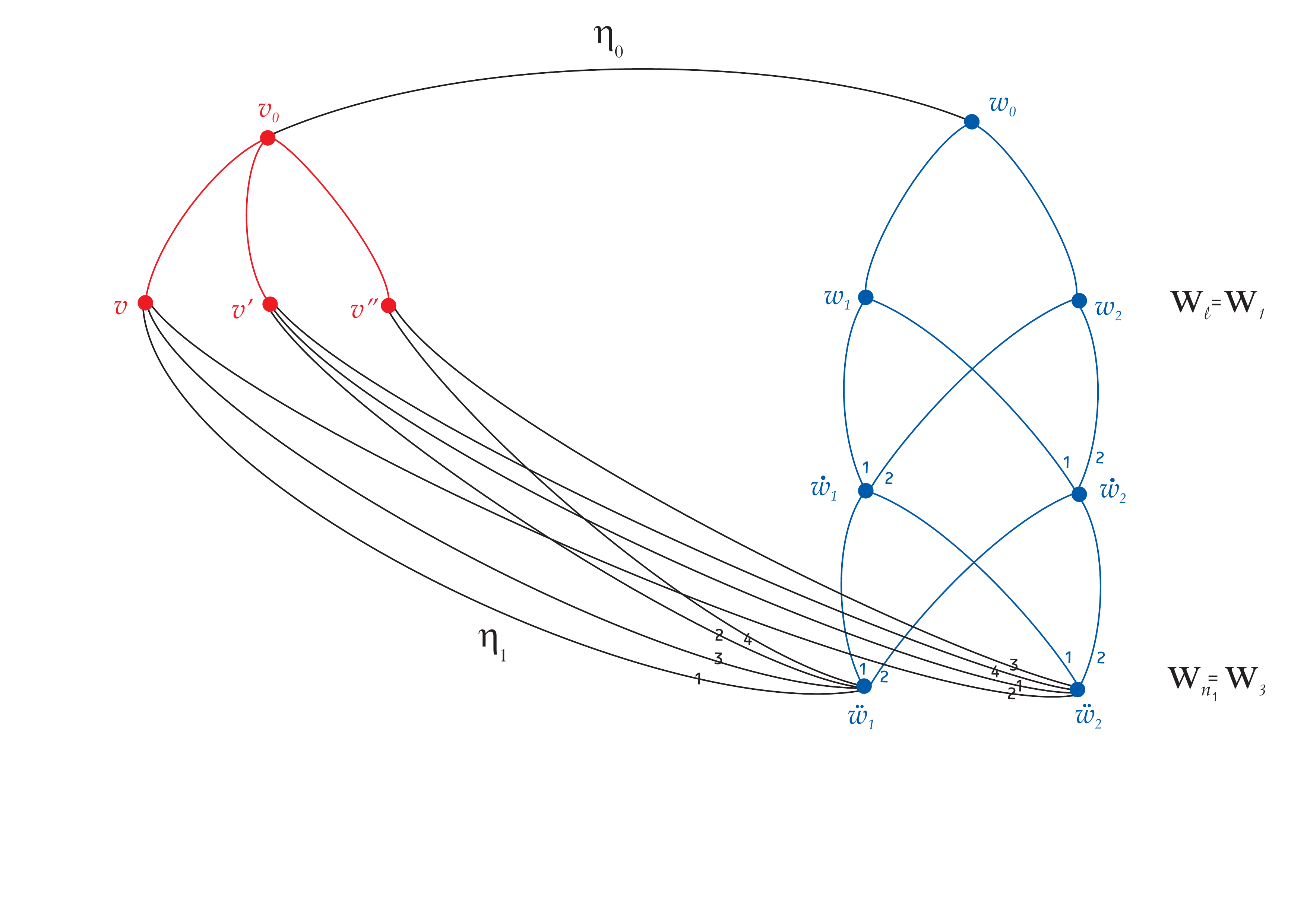}
\caption{This figure is related to the proof of Lemma \ref{firstlevel}. The words linked to the finite paths are described in Example \ref{example:one}. }\label{fig1}
\includegraphics[scale=0.2]{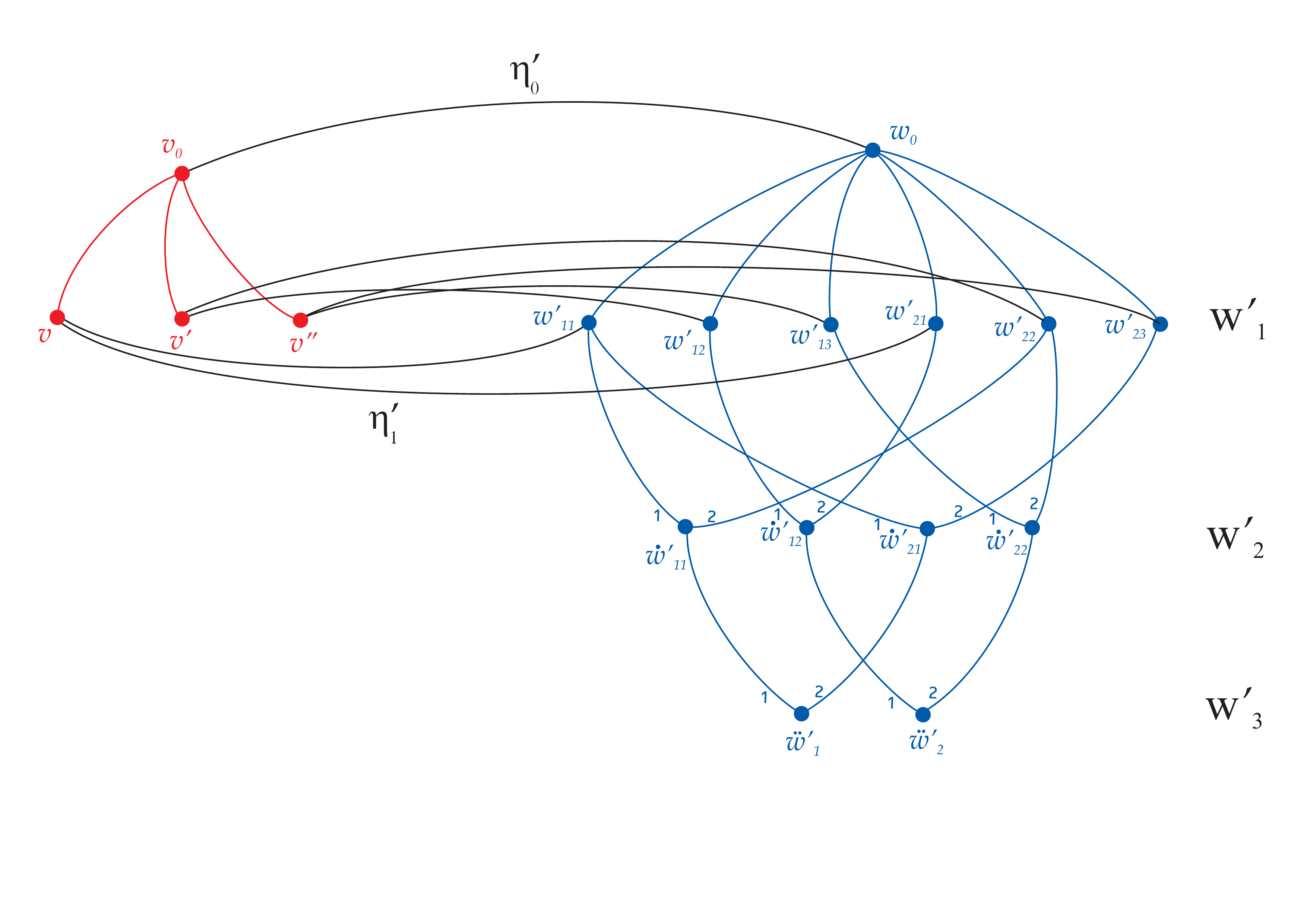}
\caption{This figure is related to the proof of Lemma \ref{firstlevel} for the diagrams in Figure \ref{fig1} described in  Example \ref{example:one}.
}\label{fig2}
\end{figure}
It remains to define
the sequence of morphisms $\eta'$. We start by defining $\eta'_1: W_1'\rightarrow  V_1^*$. Since each vertex $w'$ in $W'_1$ arises from a subword $\eta_1^{(f)}(w')\sqsubset\eta_1(w)$ for some $w\in \dot{W_{n_1}}$ and a finite path $f$ with $r(f)=w$,  we define $$\eta'_1(w'):=\eta_1^{(e)}(w).$$ 
It is clear by definition of $\eta'_1$, that for every $w'\in W'_1$ we have $\theta_1\eta'_1(w')=\eta_0\xi_{(0,\ell]}(w')=\eta'_0\xi'_1(w')$. Finally, for $i\geq 2$, set   $\eta'_i: W'_{n_i-\ell+1}\rightarrow V_i^*$ as $\eta'_i:=\eta_i$.
 \[
\xymatrix{{W'_{0}}^*\ar[d]_{\eta_{0}}
& {W'_1}^*\ar[l]_{_{\xi'_{1}}}\ar[d]_{\eta'_1}&{W'_2}^*\ar[l]_{_{\xi'_{2}}} &\cdots\ar[l]&\ \  \ {W'_{n_2-\ell+1}}^*\ar[l]_{\xi'_{n_2-\ell+1}}\ar[lld]^{\eta'_2}\\
\ V^*_{0}
&  V_{1}^*\ar[l]^{\theta_1} &V_2^*.\ar[l]^{\theta_2} 
 }
 \] 
The factor map defined by $\eta'$ is  the same as $\pi$, see \cite[Proposition 2.9]{gho} for its details. 
  \end{proof}
  \begin{example}\label{example:one}
Consider the Bratteli diagrams in Figure~\ref{fig1}.
We have $\eta_1(\ddot{w}_1)=vv'vv''$ and $\eta_1(\ddot{w}_2)=v'vv''v'$. There are four finite paths $f_1, f_2, f_3, f_4$ with range  $ \ddot{w}_1 $ and four finite paths $f'_1, f'_2, f'_3, f'_4$ with range  $\ddot{w}_2$. Moreover,  we have
 \begin{align*}\eta_1^{(f_1)}(\ddot{w}_1)&=v, \,\, \eta_1^{(f_2)}(\ddot{w}_1)=v', \,\,\eta_1^{(f_3)}(\ddot{w}_1)=v, \,\,\eta_1^{(f_4)}(\ddot{w}_1)=v'', \mbox{ and }\\ \eta_1^{(f'_1)}(\ddot{w}_2)&=v', \,\,\eta_1^{(f'_2)}(\ddot{w}_2)=v, \,\,\eta_1^{(f'_3)}(\ddot{w}_2)=v'',\,\, \eta_1^{(f'_4)}(\ddot{w}_2)=v'.
\end{align*}

Next consider the Bratteli diagrams in Figure~\ref{fig2}.
 The four words linked to $w_1$ are 
 \begin{align*}\eta_1^{(f_1)}(\ddot{w}_1)&=v, \,\,\eta_1^{(f_3)}(\ddot{w}_1)=v,\,\, \eta_1^{(f'_1)}(\ddot{w}_2)=v',\,\,  \eta_1^{(f'_3)}(\ddot{w}_2)=v''.
 \end{align*}
So $\widehat{w_1}=\{w'_{11}, w'_{12}, w'_{13}\}$ because there are only three distinct words linked to $w_1$. Similarly $\widehat{w_2}=\{w'_{21}, w'_{22}, w'_{23}\}$. There are two edges $g_1, g_2$ with source $\dot{w}_1$, namely $g_1$ with range  $\ddot{w}_1$ and $g_2$ with range  $\ddot{w}_2$, and  we have two different words linked to $\dot{w}_1$: $\eta_1^{g_1}(\dot{w}_1)=vv', \eta_1^{g_2}(\dot{w}_1)=v'v$. So $\widehat{\dot{w_1}}=\{\dot{w}'_{11}, \dot{w}'_{12}\}$. Similarly, $\dot{w}_2$ is linked to $\eta_1^{(g'_1)}(\dot{w}_2)$ and $\eta_1^{(g'_1)}(\dot{w}_2)$ which yields  $\widehat{{\dot{w_2}}}=\{\dot{w}'_{21}, \dot{w}'_{22}\}$. 
 \end{example}

\begin{remark}\label{leftmost}
 Suppose that in $B=(W,E,  \(\xi_k\)_{k\geq 1})$, each minimal  edge  at level $W_i$ has the same source  for every $i\geq 1$.
 Note that  after performing the symbol splitting procedure in Lemma~\ref{firstlevel},   the same is true for $B'$, but possibly only for  each level $W'_i$, $i\geq n_1+1$. 
\end{remark}

\begin{remark}\label{single}
Suppose that   diagram $B$ (resp. $C$)  has the property that for any vertex $w\in W_1$ (resp. $v\in V_1$) there is a single edge from $W_0$ (resp. $V_0$) terminating at $w$ (resp. $v$).  Assume that we have $n_1>1$ and  $\eta_1:W_{n_1}\rightarrow V_1$.  Using the notation in the proof of Lemma \ref{firstlevel}, in such cases we have $\ell=1$ and so $W_\ell=W_1$. Then, when the  symbol splitting is implemented on $W_1$,
 the morphism  $\eta'_1: {W'_1}^*\rightarrow V_1^*$, is  letter-to-letter.
\end{remark}
\medskip

Let us recall a result from \cite{GH} that will be used in the proof of Lemma \ref{main}.
\begin{lemma}\cite[Lemma 3.1]{GH}\label{packing1}
Let $C=(V,E,(\theta_n)_{n\geq 1})$ be an ordered Bratteli diagram for which there exist some $i\geq 1$ 
and a set of words $W\subseteq V_{i-1}^*$ such that
\begin{equation}\label{equ_packing}
\theta_{i}(v)\in W^*\ \ \text{for every}\
v\in V_{i}.
\end{equation}
Suppose  that $W$ is a minimal subset
of $V_{i-1}^*$
with respect to the inclusion
relation satisfying \eqref{equ_packing}.
Then there is an 
ordered Bratteli diagram $C'$ isomorphic to $C$, which is
constructed from $C$ 
by adding a new set of vertices, say $\mathcal V$, between   levels
 $V_{i-1}$ and $V_i$ such that $|\mathcal V|=|W|$.
\end{lemma}

We apply a restricted form of  Lemma \ref{packing1} to ERS ordered Bratteli diagrams. That is, having the above assumptions for the ERS ordered Bratteli diagram $C=(V,E,\(\theta_i\)_{i\geq 1})$, for every $i\geq 2$ we consider $$\mathcal V_i=\{\theta_i(v):\  v\in V_i\}.$$ In other words, we add a level between $V_i$ and $V_{i-1}$ if there exist  $v, v'\in V_i$, $v\neq v'$ such  that $\theta_i(v)=\theta_i(v')$. Therefore, for every $i\geq 2$, if $\mathcal V_i$ is a nonempty set we have $2\leq |\mathcal V_i|<|V_i|$. Note that the lower bound $2$ is true for all sufficiently large $i$ as we only concerned with expansive systems, and if $|\mathcal V_i|=1$ infinitely often, then the systems would be conjugate to an odometer, which is not expansive.

 Now assume that $B=(W,E,\(\xi_i\)_{i\geq 1})$ and  $C=(V,E',\(\theta_i\)_{i\geq 1})$ are two ERS properly ordered Bratteli diagrams satisfying  diagram (\ref{diagram}) for some $(\eta_i)_{i\geq 0}$.  Then one can also adjust  the morphisms $\eta_i$ for every $i\geq 2$ so that there is an analogous diagram (\ref{diagram}) for some  $(\eta_i')_{i\geq 0}$  between $B$ and $C'$, where $C'$ is as in Lemma~\ref{packing1}. In fact, when  $\mathcal V_i\neq\emptyset$ it means that there are  vertices $v, v'\in V_i$ that $\theta_i(v)=\theta_i(v')$. Therefore,  there is a vertex $\nu\in\mathcal V_i$ which is the representative of $v, v'$.  So for every $w\in W_{n_i}$, all those letters in $\eta_i(w)$ that are equal to   $v$ or $v'$ are replaced by $\nu$. 
In this way we get $\eta'_i: W_{n_i}\rightarrow \mathcal V_i$. Thus, we can assume that if $B$ and $C=(V,E',\(\theta_i\)_{i\geq 1})$ satisfy diagram (\ref{diagram}), then for each $v, v'\in V_i$ we have $\theta_i(v)\neq\theta_i(v')$.  
\medskip

The next lemma demonstrates how to optimize $n_i$ for $i \geq 2$ in relation to the sequence $(\eta_i)_{i \geq 0}$, where $\eta_i: W_{n_i} \rightarrow V_i$. The proof of this lemma shares some similarities with the proof of Lemma \ref{firstlevel}, but a key distinction is that for $i \geq 2$, the morphism $\eta_{i-1}$ does not map between two singleton sets. Due to this fact and the arguments following Lemma \ref{packing1}, optimizing $n_i$ for $i \geq 2$ does not require constructing an ordered Bratteli diagram $B'$ equivalent to $B$; instead, only $\eta_i$ for $i \geq 2$ is pushed upward.

\begin{lemma}\label{main}
Let $B=(W,E,(\xi_i)_{i\geq 1})$ and $C=(V,E',(\theta_i)_{i\geq 1})$ be two ERS properly ordered Bratteli diagrams and let $(X_B,T_B)$ and $(X_C,T_C)$ be the corresponding Vershik systems.   Let $\pi:(X_B, T_B)\rightarrow (X_C, T_C)$ be a factor map realized by $(n_i)_{i\geq 0}$  and  $\eta=\(\eta_i\)_{i\geq 0}$. For each $i\geq 1$, we can take $n_i$ to be the smallest positive integer (greater than $n_{i-1}$) such that $|\theta_{(0,i]}|$ divides $|\xi_{(0,n_i]}|$; specifically, there exists $r_i$ such that for all $w\in W_{n_i}$,  $|\xi_{(0,n_i]}(w)| = r_i|\theta_{(0,i]}|.$
\end{lemma}
\begin{proof}
 By Lemma \ref{firstlevel}, for  $\eta_1:W_{n_1}\rightarrow V_1$,
$n_1$ is the least integer  with the property   $|\xi_{(0,n_1]}|=r_1|\theta_{1}|$, $r_1\in\mathbb N$. 
Assume that $i\geq 1$; then by \eqref{diagram}, we have 
$$ |\xi_{(0,n_{i}]}|=r_i|\theta_{(0,i]} |,\ \ \ |\xi_{(0,n_{i+1}]}|=r_{i+1}|\theta_{(0,i+1]} |,\ \ {\rm for\ some}\  r_i, r_{i+1}\in\mathbb N$$
 and $$\theta_{i+1}\circ \eta_{i+1}=\eta_i\circ\xi_{(n_i,n_{i+1}]}.$$
Suppose that there exists some $\ell$, $n_i<\ell<n_{i+1}$ such that $$\exists\ t_\ell\in\mathbb N;\ \ 
 |\xi_{(0,\ell]}|=t_\ell|\theta_{(0,i+1]} |.$$   So  we have the following commutative diagram: 
\begin{equation}\label{dig4}
\xymatrix{W_{n_i}^*\ar[d]_{\eta_{i}}
& W_{\ell}^*\ar[l]_{_{\xi_{(n_i,\ell]}} }&W_{n_{i+1}}^*\ar[l]_{\xi_{(\ell,n_{i+1}]}}\ar[ld]^{\eta_{i+1}} \\
\ V_{i}^*
& V_{i+1}^*\ar[l]^{\theta_{i+1}} &  \   
 }
\end{equation}
in which for every $w\in W_{n_{i+1}}$,
\begin{equation}\label{eq1}
\eta_i\circ\xi_{(n_i,\ell]}\circ\xi_{(\ell,n_{i+1}]}(w)=\theta_{i+1}\circ\eta_{i+1}(w).
\end{equation}
Note that
for every $w\in W_{n_{i+1}}:$ $$\ r_{i+1}|\theta_{(0,i+1]}|=|\xi_{(\ell,n_{i+1}]}(w)| |\xi_{(0,\ell]}|=|\xi_{(\ell,n_{i+1}]}(w)| (t_\ell |\theta_{(0,i+1]}|),$$ 
which implies that
 $$r_{i+1}=t_\ell |\xi_{(\ell,n_{i+1}]}(w)|.$$
So for each fixed vertex $w'\in W_\ell$ and each path $e$ with $s(e)=w', r(e)=w$, there exists a subword of $\eta_{i+1}(w)\in V_{i+1}^*$ of length $t_{\ell}$ linked to $e$. As in  the proof of Lemma \ref{firstlevel}, we refer to this subword as $\eta^{(e)}_{i+1}(w')$. \\
{\bf Claim.}
 For a fixed vertex $w'\in W_\ell$, 
$$|\{\eta_{i+1}^{(e)}(w'):\  e\ {\rm is\ a\ path\ with}\ s(e)=w', r(e)\in W_{n_{i+1}}\}|=1.$$
 To prove the claim, we need to show that if  $w'\in W_\ell$ then for any two paths $e_1$ and $e_2$ with $s(e_1)=s(e_2)=w'$,  $r(e_1)=w_1$, and $r(e_2)=w_2$, with $w_1, w_2\in W_{n_1}$, we have
 \begin{equation}\label{eq2}
 \eta_{i+1}^{(e_1)}(w')=\eta_{i+1}^{(e_2)}(w').
 \end{equation}
To see this,  note that by (\ref{eq1}), $$\theta_{i+1}\eta_{i+1}^{(e_1)}(w_1)=\eta_i\xi_{(n_i,\ell]}(w')=\theta_{i+1}\eta_{i+1}^{(e_2)}(w_2).$$
Moreover, $\eta_{i+1}^{(e_j)}(w_j)\in V_{i+1}^*, j=1,2$. Now using the fact that our diagrams have the ERS property, and by the arguments following Lemma~\ref{packing1}, the equality in (\ref{eq2}) is satisfied and this proves the claim. See Figure 3 for an example. 
 \begin{figure}
\includegraphics[scale=0.4]{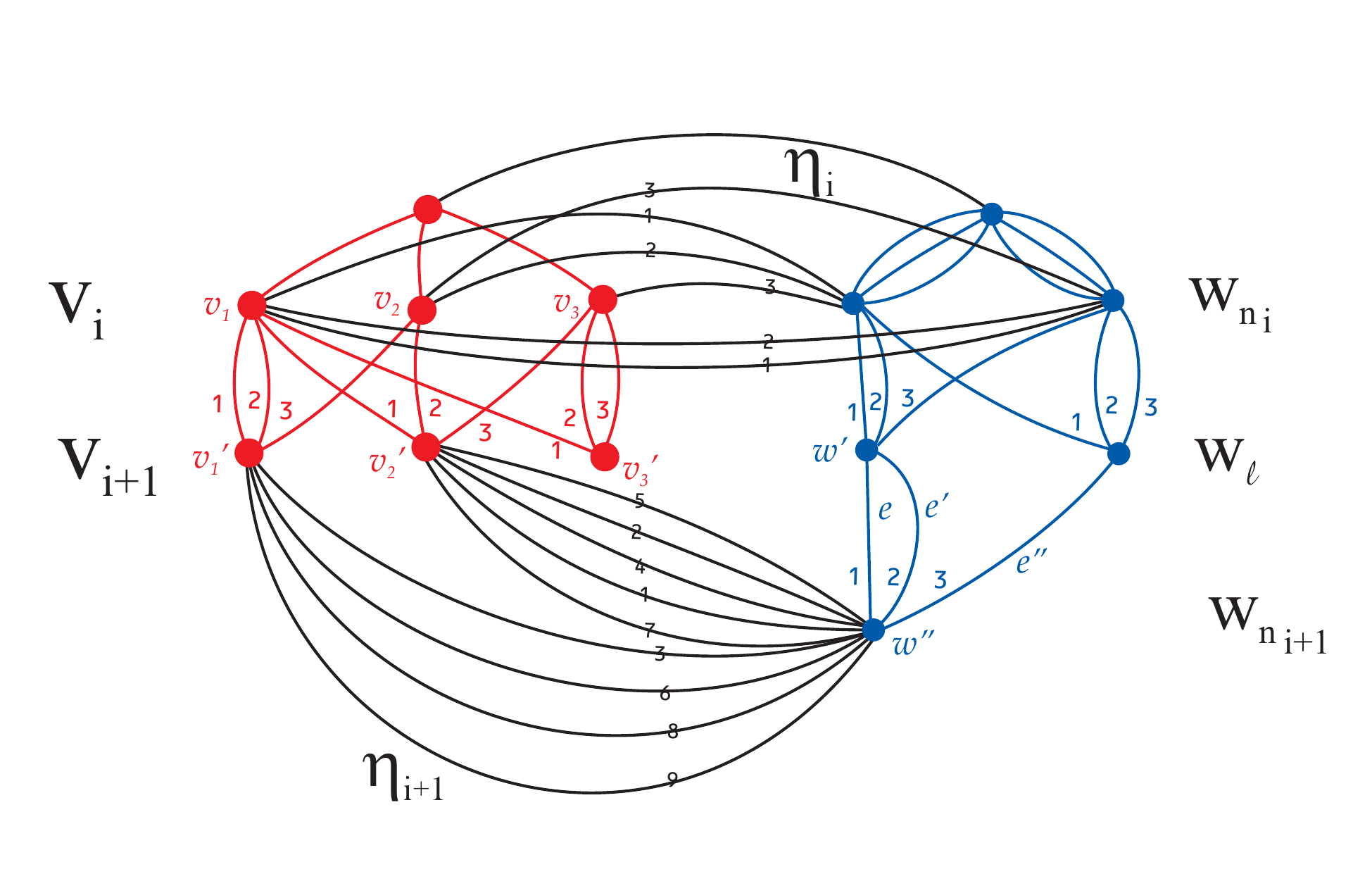}
\caption{$\eta_{i+1}(w'')=v'_2v'_2v'_1v'_2v'_2v'_1v'_2v'_1v'_1$, $\theta_{i+1}(v'_1)=v_1v_1v_2$, and $\theta_{i+1}(v'_2)=v_1v_2v_3$. Moreover, $\eta_{i+1}^{(e)}(w'')=\eta_{i+1}^{(e')}(w'')=v'_2v'_2v'_1, \eta_{i+1}^{(e'')}(w'')=v'_2v'_1v'_1$.}\label{fig3}
\includegraphics[scale=0.4]{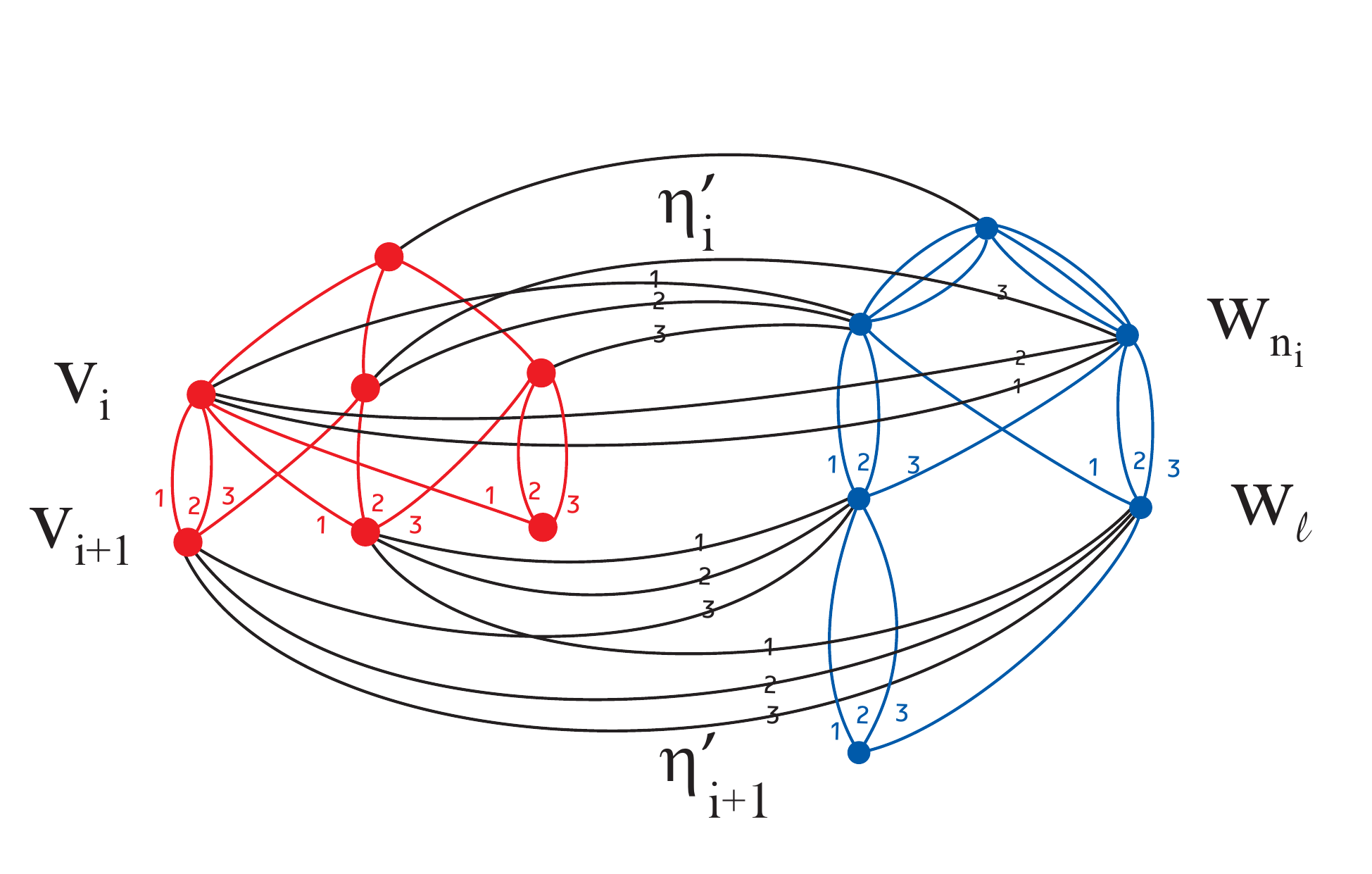}
\caption{After using the claim to push $\eta_{i+1}:W_{n_{i+1}}\rightarrow V_{i+1}$ upward, we get  $\eta'_{i+1}:W_\ell\rightarrow V_{i+1}$.}\label{fig4}
\end{figure}
Consequently, for  each vertex $w'\in W_\ell$  there exists a specific word $\alpha(w'):=\eta_{i+1}^{(e)}(w')\in V_{i+1}^*$ which appears as  subword of $\eta_{i+1}(w)$ for every $w\in W_{n_{i+1}}$ and $e$ with $s(e)=w', r(e)=w$.  Now to define 
 $\(\eta'_j\)_{j\geq 1}$ we just let  
 \begin{itemize}
 \item  $\eta'_j=\eta_j$, if $j\neq i+1$, and
 \item for  $j=i+1$, $\eta'_{j}:W_\ell\rightarrow V_{i+1}$ is defined by
 $\eta'_{j}(w')=\alpha(w')$ for every $w'\in W_\ell$.
 \end{itemize}
 The equation of commutativity  $\eta'_j\circ\xi_{(n_j,n_{j+1}]}=\theta\circ \eta'_{j+1}$  is clearly satisfied for all  $j$  by the definition of $\eta'_{i+1}$.  See Figure 4 as an example which is in continuation of Figure 3.
 \begin{figure}
\includegraphics[scale=0.4]{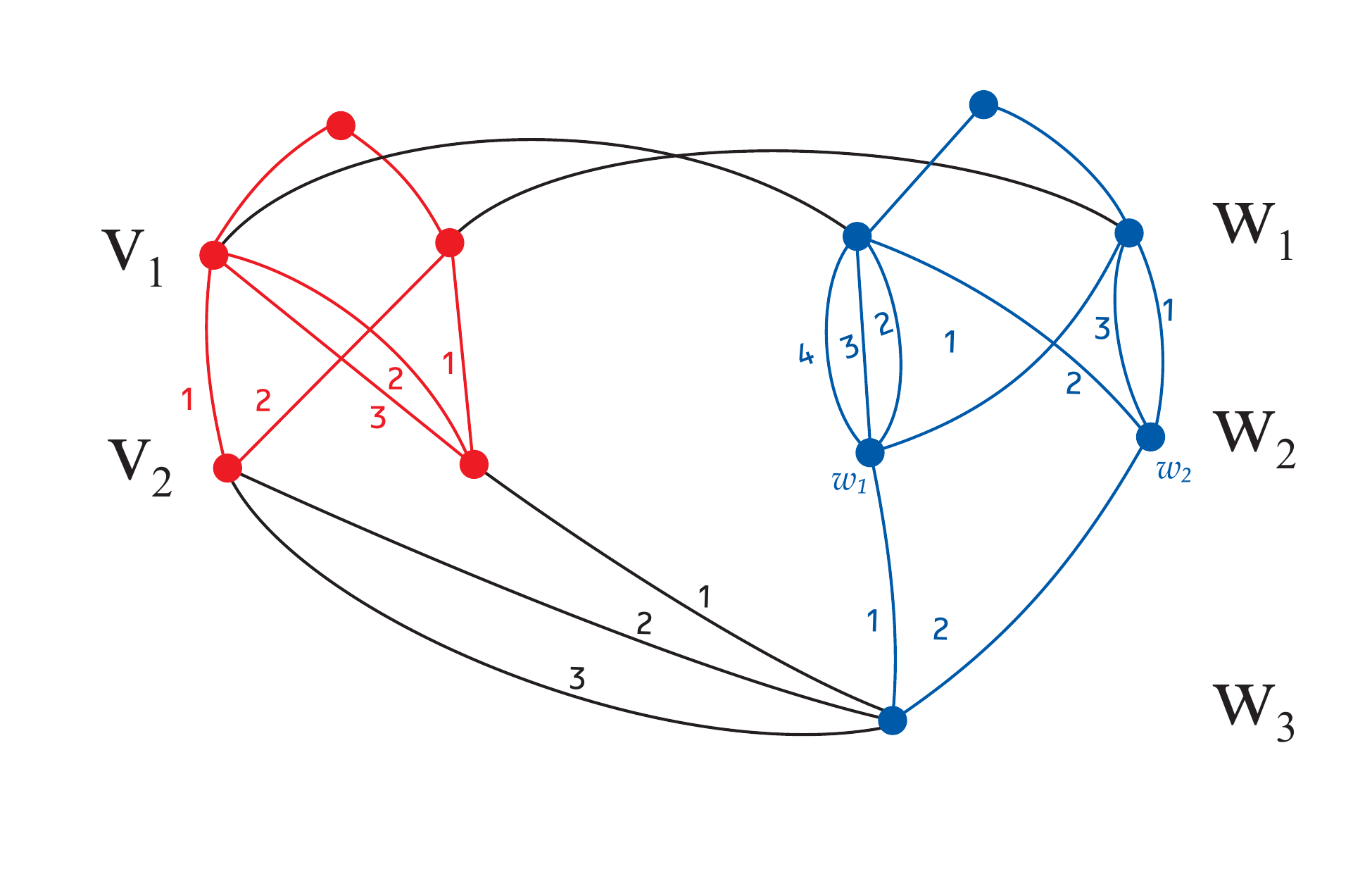}
\caption{This Figure shows the counter-example mentioned in Remark \ref{counter}.} \label{fig5}
\end{figure}
 \end{proof}
\begin{remark}\label{counter}
The condition 
$$\forall\ w\in W_{n_k},\ \ |\xi_{(0,n_i]}(w)|=r_i|\theta_{(0,i]} |$$
used in  Lemma \ref{main} for topological factorings of Vershik systems on ERS ordered Bratteli diagrams, cannot be generalized to topological factoring of Vershik systems on non-ERS Bratteli diagrams. In fact, a natural generalization of this condition for non-ERS Bratteli diagrams can be formulated as 
$$\forall \ w\in W_{n_k},\ \exists v_{i_1},\ldots, v_{i_s},\ {\rm so\ that}\  |\xi_{(0,n_i]}(w)|=\sum_{n=1}^sc_i|\theta_{(0,i]}(v_n)|,\ c_i\in\N.$$
 The  example in Figure \ref{fig5} shows that this generalized form of the condition in Lemma \ref{main} is not sufficient for pushing up the morphism $\eta_2: W_3\rightarrow V_2$ to be defined on $W_2$. This is while for both vertices $w_i\in W_2$, $i=1,2$, $|\xi_2(w_i)|$ are linear combinations of $|\theta_2(v_i)|$, $i=1,2$ 
 \end{remark}
 Let $\pi:(X_B, T_B)\rightarrow (X_C, S_C)$ be a factor map between two Vershik systems on ERS ordered Bratteli diagrams $B$ and $C$ respectively, realized by $(n_i)_{i\geq 0}$  and  $\eta=\(\eta_i\)_{i\geq 0}$. We will call $(n_i)_{i\geq 0}$ {\em optimal} if it satisfies the conditions of Lemma~\ref{main}, i.e.,  for each $i\geq 1$, $n_i$ is the least integer greater than $n_{i-1}$  such that $|\theta_{(0,i]}|$ divides $|\xi_{(0,n_i]}|$. 
\begin{lemma}\label{proper}
Let $B=(W,E,(\xi_i)_{i\geq 1})$ and $C=(V,E',(\theta_i)_{i\geq 1})$ be two ERS properly ordered Bratteli diagrams. Assume that $(X_B,T_B)$ and $(X_C,T_C)$ are the corresponding  Vershik systems and they are expansive. Suppose that  $(|\xi_{(0,i]}|)_{i\geq 1}$ (resp. $(|\theta_{(0,i]}|)_{i\geq 1}$) as constructive period structure of $x_{\min}$ (resp.  $y_{\min}$).   If $\pi:(X_B,T_B)\rightarrow (X_C,T_C)$  is a  factor map with $\pi(x_{\min})=y_{\min}$ that is  realized by an optimal $(n_i)_{i\geq 0}$ and $\(\eta_i\)_{i\geq 0}$, then there exists $i_0$ such that 
for every $ i\geq {i_0},\ \ n_{i+1}=n_i+1$.
\end{lemma}
\begin{proof}
We prove that for every $i\geq 1$,  if the sources of the minimal edges with range $W_{n_i+1}$ (resp. $V_{i+1}$) are the same vertex $w_1^{n_i}$ (resp. $v_1^i$), existence of  $\eta_i: W_{n_i}^*\rightarrow V_i^*$ will lead to  
$$|\theta_{(0,i+1]}| \ {\rm divides}\ |\xi_{(0,n_i+1]}|,$$
which implies that $\eta_{i+1}:W_{n_i+1}\rightarrow V_{i+1}$ and so $n_{i+1}=n_i+1$. Note that the ordered Bratteli diagram guaranteed by Theorem~\ref{auto} satisfies this property on the minimal edges at all levels. However, in assuming that our sequence $(n_i)_{i\geq 0}$ is optimal, we may have had  to modify the diagram as in Lemma \ref{firstlevel}, in which case we may lose this condition on a finite number of levels. 
So  let $i_0\geq 1$ be the minimum $i$ such that the minimal edge terminating at every vertex in $W_k$ has the same source for every  $k\geq n_i$. 
Now for every  $i\geq i_0$ and $\eta_i:W_{n_i}^*\rightarrow V_i^*$, since $\pi({x_{\min}})={y_{\min}}$, we have
\begin{equation}\label{proper1}
 \eta_i(w^{n_i}_1)=v^i_1.
\end{equation}
Consider the following diagram.
\begin{equation}\label{dig5}
\xymatrix{\ W_0\ar[d]_{\eta_0}&W_{1}\ar[l]_{\xi_1}&W^*_{n_i}\ar[d]_{\eta_{i}}\ar[l]_{\xi_{(1,n_i]}}
& \ W_{n_i+1}^*\ar[l]_{_{\xi_{n_i+1}}}&\ \ \ W_{n_j}^*\ar[l]_{_{\xi_{(n_i+1,n_j]}}}\ar[ld]^{\eta_{j}}  \\
V_0&\ar[l]^{\theta_1}V_1^*&V^*_{i}\ar[l]^{\theta_{(1,i]}}
&  V_{j}^*\ar[l]^{\theta_{(i,j]}} & 
 }
 \end{equation}
 The image of every vertex $w\in W_{n_i+1}$ under the morphism
$\eta_i\xi_{n_i+1}$ belongs to $V_i^*$ and by the minimal edges condition, its prefix is $v^i_1$. In other words,
\begin{equation}\label{key1} 
\forall\ w\in W_{n_i+1}:\ \exists u\in V_i^*,\ \eta_i(\xi_{n_i+1}(w))=v_1^iu.
\end{equation} 
 Moreover,  by commutativity of diagram (\ref{dig5}),
\begin{equation}\label{dig6}
\forall w\in W_{n_i+1}:\   |\theta_{(0,i]}(\eta_i(\xi_{{n_{i}+1}}(w)))|=|\xi_{(0,n_i+1]}|.
\end{equation}
Indeed,  for every $j>i$:
\begin{equation}\label{dig7}
\forall w\in W_{n_j}:\  \theta_{(0,i]}(\eta_i(\xi_{(n_i,n_j]}(w)))=\theta_{(0,j]}(\eta_j(w)). 
\end{equation}
 Consider  the truncation map $\tau_{n_j}:X_B\rightarrow {\mathcal P_{n_j}}$, and the intermediate  shift $(X_{n_j},\sigma)$ as defined at the end of Section~\ref{versh}, and the point 
$$\tilde\tau_{n_j}(x_{\min})=(\tau_{n_j}(T_B^n(x_{\min})))_{n\geq 0}\in X_{n_j}.$$
Theorem \ref{inversesys} gives us the following commutative diagram:
{\small  \begin{equation}\label{dig8}
\xymatrix{(X_1,\sigma)\ar[d]_{\pi_{1}}
&(X_{n_i},\sigma)\ar[l]_{\tilde\tr_{(1,n_i]}}\ar[d]_{\pi_{i}}&\ (X_{n_i+1},\sigma)\ \  \ar[l]_{\tilde\tr_{n_i+1}}&\ \ (X_{n_j},\sigma)\ar[l]_{\tilde\tr_{(n_i+1,n_j]}}\ar[d]_{\pi_j}& (X_B,T_B,x_{\min})\ar[l]_{\tilde\tau_{n_j}}\ar[d]_{\pi} \\
(Y_{1},\sigma)
&(Y_i,\sigma)\ar[l]^{\tilde\tr_{(1,i]}}&\ (Y_{i+1},\sigma)\ \ \ar[l]^{\tilde\tr_{i+1}}&\ \ (Y_j,\sigma)\ar[l]^{\tilde\tr_{(i+1,j]}}&(X_C,T_C,y_{\min})\ar[l]_{\tilde\tau_{j}}\   
 }
\end{equation}
}
Here we assumed without loss of generality that $|\theta_1|=|\xi_1|=1$ (otherwise, we can perform a symbol splitting to achieve this). 
So  the word:
$${\bf w}=\tau_{n_i}(x_{\min})\tau_{n_i}(T_B(x_{\min}))\cdots \tau_{n_i}(T_B^{|\xi_{(0,n_i]}|-1}(x_{\min}))=\xi_{(1,n_i]}(w^{n_i}_1)$$ appears periodically in $\tilde\tau_{n_j}(x_{\min})$ with period $|\xi_{(0,n_i+1]}|$. More precisely,  for every $m\geq 0$ and $n=m|\xi_{(0,n_i+1]}|$, the word ${\bf wu_{n}}$,
in which $\bf u_{n}$ is the minimal edge with source $w_1^{n_i}\in W_{n_i}$ that terminates at a vertex in $W_{n_i+1}$ which caries $T_B^{n}(x_{\min})$,  is the initial word of
$\tilde \tau_{n_i+1}(T_B^{n}(x_{\min}))$. Then for  every $j>i$,
the word ${\bf wu_{n}u'_{n}}$ appears as the initial word of $\tilde\tau_{n_j}(T_B^{n}(x_{\min}))$,
where $\bf u'_{n}$ is the finite path from $W_{n_i+1}$ to the vertex in $W_{n_j}$ which carries $x_{\min}$.
\smallskip

Now consider ${\bf v}=\theta_{(1,i]}(v_1^i)\sqsubset\eta_i({\bf w})$ of length $|{\bf v}|=|\theta_{(0,i]}|$, defined as
 $${\bf v}=\tau_{i}(y_{\min})\tau_{i}(T_C(y_{\min}))\cdots \tau_{i}(T_C^{|\theta_{(1,i]}|-1}(y_{\min})).$$
By (\ref{key1})-(\ref{dig6}) and commutativity of the diagram  \ref{dig8},  for every $j>i$, the word $\bf v$ appears periodically in  $\tilde\tau_{j}(y_{\min})$ with period $|\xi_{(0,n_i+1]}|$.  
\noindent
  But by the assumptions on $C$, $y_{\rm min}$ has period structure $(|\theta_{(0,n]}|)_{n\geq 1}$ and since the period structure is constructive,  the essential period of the word $\bf v$ to appear in $\tilde\tau'_j(y_{\rm min})$ is equal to $|\theta_{(0,i+1]}|$. Consequently, $|\theta_{(0,i+1]}|$ divides $|\xi_{(0,n_i+1]}|$.  
\end{proof}
Now we have the tools to study topological factoring between two Toeplitz shifts.  
We first have the following proposition as a direct corollary  of Lemma \ref{proper}.
\begin{proposition}\label{general1}
Let $(X,\sigma)$ and $(Y, \sigma)$ be two  Toeplitz shifts. Suppose  that  $x$ and $y$ are Toeplitz sequences in $X$ and $Y$ with constructive period structures $(p_i)_{i\geq 1}$  and $(q_i)_{i\geq 1}$,  respectively. If there exists a topological factor map $\pi:(X,\sigma)\rightarrow (Y,\sigma)$ with $\pi(x)=y$, then there exists $i_o\geq 1$ such that for every $i\geq 1$, $q_i\mid p_{i_0+i}$. 
\end{proposition}
 The proof of the next proposition is based on  the main point of the arguments used in the proof of Lemma \ref{proper}.
 \begin{proposition}\label{general2}
Let $(X,\sigma)$ and $(Y, \sigma)$ be two  Toeplitz shifts with $x_0\in X$ as a Toeplitz sequence. If there exists a topological factor map $\pi:(X,\sigma)\rightarrow (Y,\sigma)$, then $\pi(x_0)$   is a Toeplitz sequence in $Y$.
\end{proposition}
\begin{proof}
Let $ y_0 = \pi(x_0)$. Consider a period structure for $x_0$ with the corresponding constructive period structure $(p_i)_{i \geq 1}$, which leads to the  Bratteli-Vershik realization $(X_B,T_B)$ of $(X,\sigma,x_0)$  through the properly ordered ERS Bratteli diagram $B = (W, E, (\xi_i)_{i \geq 1})$ where $|\xi_i| = p_i$. Let $(X_C,T_C)$ be any properly ordered and simple Bratteli diagram associated with $(Y, \sigma, y_0)$ through the Bratteli diagram $C = (V, E', (\theta_i)_{i \geq 1})$  where $|\theta_{(1,i]}(v_1^i)|$ is strictly increasing. By Theorem \ref{inversesys}, there exist a sequence $(n_i)_{i\geq 1}$ and a sequence of morphisms $(\eta_i)_{i \geq 0}$  that realizes the factor map $\pi$ between the two Bratteli diagrams. The arguments from the proof of Lemma \ref{proper}, specifically from equation \ref{proper1} to the second line of the final paragraph, shows that for every  $i\geq 1$, the word ${\bf v_i}:=\theta_{(1,i]}(v_1^i)$  appears in $\tilde\tau_i(y_{\min})$  periodically with period $p_{n_i}$. This implies that in $(Y_i,\sigma)$, the point $\tilde\tau_i(y_{\min})$ is quasi-periodic.  As $(Y,\sigma)$ and so $(X_C,T_C)$ is not  odometer, for sufficiently large $i\geq 1$, $(Y_i,\sigma,\tilde\tau_i(y_{\min}))$ is topologically conjugate to $(X_C,T_C,y_0)$. Moreover, $\tilde\tau_i(y_{\min})$ is not periodic while it is quasi-periodic. So it  is a Toeplitz sequence.
\end{proof}
 
\smallskip
\noindent{\bf Proof of Theorem \ref{thm:main}.}
 Suppose that $\pi:(X,\sigma)\rightarrow (Y,\sigma)$ with $\pi(x)=y$  exists and $q\nmid p$.  Note that the period structure $(q^i)_{i\geq 1}$  defines the maximal equicontinuous factor of $(Y,\sigma)$, which must be  an equicontinuous factor of $(X,\sigma)$. So we write $ p=p_1^{\ell_1}\cdots p_k^{\ell_k}$  and $q=p_1^{\ell'_1}\cdots p_k^{\ell'_k}$ and assume that  for every $1\leq j\leq k$,  $\ell_j, \ell'_j> 0$.  Consider the Bratteli-Vershik realizations $(X_B,T_B)$ and $(X_C,T_C)$ of the two systems through ERS proper ordered Bratteli diagrams $B=(W,E,(\xi_i)_{i\geq 1})$ and $C=(V,E',(\theta_i)_{i\geq 1})$ respectively, so that $x_{\min}\in X_B$ and $y_{\min}\in X_C$ are associated to $x$ and $y$ respectively. So
 $$|\xi_{(0,i]}|=(p_1^{\ell_1}\cdots p_k^{\ell_k})^i,\ \ {\rm and}\ \ |\theta_{(0,i]}|=(p_1^{\ell'_1}\cdots p_k^{\ell'_k})^i.$$
By Proposition \ref{factor1} and Lemma \ref{proper} and using their notation,  existence of factoring $\pi$  between the two systems with $\pi(x)=y$ implies that there exist sequences $(n_i)_{i\geq 0}$,  $\eta=\(\eta_i\)_{i\geq 0}$ and $i_0\geq 1$ so that for every  $i\geq 0$, $\eta_{i_0+i}:W_{n_{i_0}+i}^*\rightarrow V_i^*$ and  the following diagram 
\begin{equation}\label{local}
\xymatrix{W_{n_{i_0}+i}^*\ar[d]_{\eta_{i_0+i}}
& \ \ W_{n_{i_0}+i+1}^*\ar[l]_{_{\xi_{{n_{i_0}+i+1}}}}\ar[d]_{\eta_{i_0+i+1}} &\ \  \\
\ V_{i_0+i}^*
& \ \ V_{i_0+i+1}^*\ar[l]^{\theta_{i_0+i+1}} &\   
 }
\end{equation}
Commutes. Thus, 
\begin{equation}\label{division}
 \forall i\geq 1,\ |\theta_{(0,i_0+i]}| \ {\rm divides}\  |\xi_{(0,n_{i_0}+i]}|.  
 \end{equation}
 Since $q\nmid p$, there exists  $1\leq s\leq k$ such that
 $\ell_s/\ell's<1$. 
Now by  (\ref{division}), for every $i\geq 1$,
$$(\ell_sn_{i_0}-\ell'_si_0)+(\ell_s-\ell'_s)i\geq 0\ \Rightarrow \ (\ell_s-\ell'_s)\geq -(\ell_sn_{i_0}-\ell'_si_0)/i,\ \forall i\geq 1.$$
Note that $\ell_sn_{i_0}-\ell'_si_0$  is a positive constant.
Therefore,
$\ell_s\geq \ell'_s$ which is a contradiction. 
More precisely, by Proposition \ref{factor1}, existence of $\pi$ implies that for every $i\geq 1$ there exists $j\geq i$ such that we have a morphism $\eta_{i_0+i}:W_{n_{i_0}+j}^*\rightarrow V_{i_0+i}^*$. Moreover, Lemma \ref{proper} ensures that  $j=i$. But  $q\nmid p$  enforces that $j>i$.
\qed
\medskip

Restricted to substitutional Toeplitz shifts, Theorem~\ref{thm:main} is related to a topological version of Cobham's theorem. See Subsection \ref{subs} for the notation.
\smallskip

 \begin{corollary}\label{Cobham}
Let $u\in A^\N$ be  $p$-automatic, generated by $(\tau,\theta)$, and let $(X,\sigma)$ be the shift generated by $u$.  Suppose that   $\theta$ has a coincidence and is primitive and aperiodic. If $p$ and $q$ are multiplicatively independent and $u$ is  also $q$-automatic, then $X$ is finite. 
\end{corollary}
\begin{proof}
If $u$ does not have a finite orbit under the shift, then its shift-orbit closure $(X_u,\sigma)$ must be topologically conjugate to a  length-$p$ substitution shift $(X_{\eta_p},\sigma)$, by \cite[Theorem 22]{MY}. Furthermore, as the property of being a Toeplitz shift is a topological conjugacy invariant, $\eta_p$ must have a coincidence. We work with the pure base $\zeta_p$  of $\eta_p$, i.e.,  $X_{\eta_p}$ is a height $h_p$ suspension over $X_{\zeta_p}$ for some $h_p$ coprime to $p$,          and there exists $k$, an $0\leq i\leq p^{k}-1$, and a letter $b\in A_\zeta$ such that for all $a\in A_\zeta$, the $i$-th letter of $\zeta_p^k(a)$ equals $b$. Let $i = i_{k-1}p^{k-1}+ i_{k-2}p^{k-2} + \dots +i_0$ where each $i_j\in \{ 0 , \dots p-1\}$. There is a Toeplitz sequence $x$ in $X_{\zeta_p}$ which, under the maximal equicontinuous factor map $\pi_p: X_{\zeta_p}\rightarrow \Z_p$ satisfies $\pi_p^{-1}(z) = \{x \}$, where $z$ is the periodic sequence defined by the word $i_0 i_1 \dots i_{k -1}$. For this point $x\in X_{\zeta_p}$, there exists $ k_p\geq 1$ such that  $x$ has a constructive period structure ${\bf p}=( p^{ k_p n})_{n\geq 1}$. 

 By the same reasoning, if $u$ were also $q$-automatic, then $(X_u,\sigma)$ is also conjugate to a length-$q$ substitution shift $(X_{\eta_q},\sigma)$ with  a coincidence, and with pure base $\zeta_q$  so that $X_{\eta_q}$ is a height $h_q$ suspension over $X_{\zeta_q}$ for some $h_q$ coprime to $q$.
  Since  $(X_{\eta_p}, \sigma)$ and $(X_{\eta_q}, \sigma)$ are each conjugate to $(X_u,\sigma)$, they are conjugate to each other.  Recall that we can assume that $p$ and $q$ have the same prime divisors. Therefore $(X_{\zeta_p}, \sigma)$ and $(X_{\zeta_q}, \sigma)$ are also conjugate to each other via $F:X_{\zeta_p}\rightarrow X_{\zeta_q}$.
  In particular, they have the same maximal equicontinuous factor,  and so one can apply the argument of  \cite[Proposition 3.24]{CQY} to conclude that $F(x)$ is also a Toeplitz sequence that lives in a fibre of the equicontinuous factor which is  defined by an eventually periodic point $y$, and which has a  constructive period structure$(q^{k_qn})_{n\geq 1}$ for some $k_q$. If $p$ and $q$ are multiplicatively independent, then $p^{k_p}$ and $q^{k_q}$ are as well. Now Theorem~\ref{thm:main} gives a contradiction. 
\end{proof}
\medskip

%%%%%%%%%%%%%%%%%%%%%%%%%%%%%%%%%%%%%%%%%%%%%%%%%%%%%%%%%
%%%%%%%%%%%%%%%%%%%%%%%%%%%%%%%%%%%%%%%%%%%%%%%%%%%%%%%%%%%%%

{\bf Acknowledgements.} 
 The research of the  authors  was  supported by the EPSRC grant number EP/V007459/2.

%%%%%%%%%%%%%%%%%%%%%%%%%%%%%%%%%%%%%%%%%%%%%%%%%

\end{document}